\documentclass [12pt,a4paper,reqno]{amsart}
 \input amssymb.sty
\usepackage{mathabx}

\input xy
\xyoption{all}

\def\vrp{\varphi}
\def\m{m}

\def\vo{vo}
\def\VO{$V^0$}

\def\Sep{\operatorname{Sep}}
\newcommand\D[1]{\operatorname{D{#1}}}

\def\chv{\widecheck{v}}

\newcommand{\ds}[1]{\ {#1} \ }
\newcommand{\dss}[1]{\quad {#1} \quad }

\def\sm{\setminus}
\def\inu{\nu^{-1}}
\def\onto{\twoheadrightarrow}

\def\tot{\operatorname{tot}}
\def\getot{\geq_{\tot}}
\def\letot{\leq_{\tot}}

\def\dcup{ \; \dot \cup \;}

\newcommand{\etype}[1]{\renewcommand{\labelenumi}{(#1{enumi})}}

\def\cT{\mathcal T}
\def\cG{\mathcal G}

\def\STR{\operatorname{STR}}
\def\OSTR{\operatorname{OSTR}}

\def\tT{\mathcal T}

\def\tG{\mathcal G}

\def\endbox{ \hfill\quad\qed}

\def\EO{\operatorname{EO}}
\def\Gh{\operatorname{Gh}}

\def\eroman{\etype{\roman}}
\def\ealph{\etype{\alph}}

\def\pSkip{\vskip 1.5mm \noindent}

\def\Nil{\operatorname{Nil}}
\def\Vz{V^0}
\def\N{\mathbb N}

\def\mfA{\mathfrak A}
\def\mfq{\mathfrak q}
\def\mfp{\mathfrak p}
\def\mfa{\mathfrak a}

\def\darM{M_{\downarrow}}
\def\uarM{M^{\uparrow}}
\def\darv{v_{\downarrow}}
\def\uarv{v^{\uparrow}}

\def\olM{\overline M}

\def\al{\alpha}
\def\bt{\beta}
\def\gm{\gamma}

\newtheorem{thm}{Theorem} [section]
\newtheorem*{thm*}{Theorem}
\newtheorem{cor}[thm]{Corollary}
\newtheorem{lem}[thm]{Lemma}
\newtheorem{lemma}[thm]{Lemma}
\newtheorem{prop}[thm]{Proposition}

\newtheorem*{claim*} {Claim}
\newtheorem*{theorem4.6'} {Theorem 4.6$'$}
\newtheorem{acknowledgment*}[thm] {Acknowledgment}
\newtheorem{example}[thm]{Example}
\newtheorem{sexample}[thm]{Subexample}
\newtheorem{examp}[thm]{Example}

\newtheorem{examples}[thm]{Examples}
 \newtheorem{rem}[thm]{Remark}
 \newtheorem{remark}[thm]{Remark}
  \newtheorem{remarks}[thm]{Remarks}
 \newtheorem*{remark*}{Remark}
 \newtheorem{defn}[thm]{Definition}
\newtheorem{construction}[thm]{Construction}

\newtheorem{schol}[thm]{Scholium}

\newtheorem{notations}[thm]{Notations}

\newtheorem*{notation*} {Notation}
\newtheorem*{notations*} {Notations}

\newtheorem*{comment*} {Comment}

\newcommand{\thmref}[1]{Theorem~\ref{#1}}
\newcommand{\propref}[1]{Proposition~\ref{#1}}

\newcommand{\lemref}[1]{Lemma~\ref{#1}}

 \renewcommand{\sectionmark}[1]{}

\newcommand{\bfem}[1]{\textbf{#1}}

\newcommand{\diag}{\operatorname{diag}}

 \newcommand{\id}{\operatorname{id}}

 \newcommand{\supp} {\operatorname{supp}}

\newcommand{\osr}{\overset\sim \rightarrow}

\begin{document}

\title[Monoid valuations and value ordered supervaluations ]
{Monoid valuations \\ \vskip 2mm and value ordered supervaluations
}
\author[Z. Izhakian]{Zur Izhakian}
\address{Department of Mathematics, Bar-Ilan University, Ramat-Gan 52900,
Israel}
\email{zzur@math.biu.ac.il}
\author[M. Knebusch]{Manfred Knebusch}
\address{Department of Mathematics,
NWF-I Mathematik, Universit\"at Regensburg 93040 Regensburg,
Germany} \email{manfred.knebusch@mathematik.uni-regensburg.de}
\author[L. Rowen]{Louis Rowen}
 \address{Department of Mathematics,
 Bar-Ilan University,  Ramat-Gan 52900, Israel}
 \email{rowen@macs.biu.ac.il}

\thanks{The research of the first author has been supported  by the
Oberwolfach Leibniz Fellows Programme (OWLF), Mathematisches
Forschungsinstitut Oberwolfach, Germany.}

\thanks{The research of the first and third authors have  been  supported  by the
Israel Science Foundation (grant No.  448/09).}

\thanks{The research of the second author was supported in part by
 the Gelbart Institute at
Bar-Ilan University, the Minerva Foundation at Tel-Aviv
University, the Department of Mathematics   of Bar-Ilan
University, and the Emmy Noether Institute at Bar-Ilan
University.}



\date{\today}




\begin{abstract}
We complement two papers on supertropical valuation theory
(\cite{IKR1}, \cite{IKR2}) by providing natural examples of
\m-valuations (= monoid valuations), after that of supervaluations
and transmissions between them. The supervaluations discussed have
values in \textbf{totally ordered supertropical semirings}, and
the transmissions discussed respect the orderings. Basics of a
theory of such semirings and transmissions are developed as far as
needed.
\end{abstract}

\maketitle

\tableofcontents

\baselineskip 14pt

\numberwithin{equation}{section}

\section*{Introduction}

The present paper is a complement to  the papers \cite{IKR1} and
\cite{IKR2} on supertropical valuation theory by the same authors.
We deal with semirings which always are taken to be commutative.
Generalizing Bourbaki's notion of a valuation on a commutative
ring \cite{B}, we introduced in \cite{IKR1} \emph{\m-valuations}
(= monoid valuations) and then \emph{supervaluations} on a
(commutative) semiring $R$. These are certain maps from $R$ to a
``\emph{bipotent semiring}'' $M$ and a ``\emph{supertropical
semiring}'' $U$, respectively.

To repeat, a semiring $M$ is \textbf{bipotent} if $M$ is a totally
ordered monoid  under multiplication with smallest element $0$,
and the addition is given by $x+ y = \max(x,y)$. Then an
\textbf{\m-valuation} on $R$ is a multiplicative map $v: R \to M
$, which sends $0$  to $0$, $1$ to $1$, and obeys the rule $v(a+b)
\leq v(a) + v(b).$ We call $v$ a \textbf{valuation} if moreover
the semiring $M$ is cancellative. \{In the classical case of a
Krull valuation $v$, $R$ is a field and $M = \tG \cup \{0\}$, with
$\tG$ the value group of $v$ written in multiplicative notation.\}

A \textbf{supertropical semiring} $U$ is a  semiring such that
$e:= 1+1$ is an idempotent of $U$ and two more axioms hold
(\cite[Definitions 3.5 and 3.9]{IKR1}), which imply in particular
that the ideal $M := eU$ is a bipotent semiring. The elements of
$M \sm \{ 0 \}$ are called \textbf{ghost} and those of $\tT(U) :=
U \sm M$ are called \textbf{tangible}. The zero element of $U$ is
regarded both as ghost and tangible. For $x \in U$ we call $ex$
the \textbf{ghost companion} of $x$. For $x,y \in U$ we have the
rule
$$ x + y = \left\{
\begin{array}{llll}
  y  &  & \text{if} &  ex < ey, \\[1mm]
  x  &  & \text{if} &  ex > ey, \\[1mm]
  ex  &  & \text{if} &  ex = ey. \\
\end{array}
\right.$$ Thus  addition on $U$ is uniquely defined by
multiplication and the element $e$.  We also mention that $ex = 0
$ implies $x = 0$. We refer to \cite[\S3]{IKR1} for all details.

Finally, a \textbf{supervaluation} on $R$ is a multiplicative map
$\vrp: R \to U$ to a supertropical semiring $U$ sending $0$ to $0$
and $1$ to $1$, such that the map $e \vrp : R \to eU$, $a \mapsto
e \vrp(a)$, is an \m-valuation. We then say that $\vrp$
\textbf{covers} the \m-valuation  $v:= e\vrp$.

If $\vrp: R \to U$ is a  supervaluation  then $U' := \vrp(R) \cup
e \vrp(R)$ is a sub-semiring of $U$ and is again supertropical. In
practice we nearly always may replace $U$ by $U'$ and then have a
supervaluation at hand which we call
\textbf{surjective}\footnote{Although this does not mean
surjectivity in the usual sense, there is no danger of confusion
since a supervaluation $\vrp:R \to U$ hardly  ever can be
surjective as a map except in the degenerate case $U=M$.}.

Given a surjective supervaluation $\vrp: R \to U$ and a map $\al:U
\to V$ to a supertropical semiring $V$, the map $\al \circ \vrp$
is again a supervaluation iff $\al$ is multiplicative, sends $0$
to $0$, $1$ to $1$, $e$ to $e$, and restricts to a semiring
homomorphism from $eU$ to $eV$. \{We denote the elements $1+1$ in
$U$ and $V$ both by ``$e$''.\} We call such a map $\al: U \to V$ a
\textbf{transmission}. Any semiring homomorphism from $U$ to $V$
is a transmission, but there exist others.

 Transmissions are tied up with the relation of
\textbf{dominance} defined in \cite[\S5]{IKR1}. If $\vrp: R \to U
$ and $\psi: R \to V$ are supervaluations and $\vrp$ is
surjective, then $\vrp$ \textbf{dominates} $\psi$, which we denote
by $\vrp \geq \psi$, iff there exists a transmission $\al: U \to
V$ with $\psi =\al \circ \vrp$. If $\vrp \geq \psi$ we also say
that $\psi$ is a \textbf{coarsening} of the supervaluation $\vrp$.

A bipotent semiring $M$ may be viewed as a supertropical semiring
$U$ with empty set $\tT(U)$, i.e., $U = eU =M$. Then a
transmission $\gm: M \to N$ is just a semiring homomorphism.  In
other terms, $\gm$  is an order preserving monoid homomorphism
with $\gm(0) =0$. If $R$ is a field and $v: R\to M$, $w: R \to N$
are Krull valuations (in multiplicative notation), then the
dominance relation $v \geq w$  means that $w$ is a coarsening of
$v$ in the classical sense.

At crucial points in the paper \cite{IKR1}, \cite{IKR2} we had to
assume that the supervaluations in question cover a valuation $v:
R \to M$ instead of just an \m-valuation, i.e., $M$  had to be
assumed cancellative. On the other hand these papers contain few
examples of true \m-valuations. Thus a reader might suspect that
it is better in supertropical valuation theory to focus from the
beginning on valuations instead of \m-valuations. The first goal
of the present paper is to clarify this situation,

In \S\ref{sec:1} we study two very natural classes of
\m-valuations, the so-called \textbf{$V$-valuations} and
\textbf{$\Vz$-valuations}. They have been introduced  (on rings)
by Harrison-Vitulli \cite{HV1} and D.~Zhang \cite{Z},
respectively. To our opinion these \m-valuations, which often are
not valuations, have not yet found the attention in the literature
that  they deserve.

In \S\ref{sec:1} it is proved that every nontrivial \m-valuation
dominates both a $V$-valuation and a $\Vz$-valuation (which may be
different) in a canonical way. There are also given various
instances of dominance $v \geq w$  with $v$ a $\Vz$-valuation and
$w$ a $V$-valuation, or vice versa. Then in \S\ref{sec:2} we
exhibit a canonical way to coarsen a given \m-valuation to a
valuation. If $v$ is a $V$-valuation or a $\Vz$-valuation, almost
always this coarsening is again a $V$-valuation or a
$\Vz$-valuation. One gets the impression that a supertropical
valuation theory excluding \m-valuations would be very incomplete.

A second goal of the paper is to give natural explicit examples of
supervaluations and dominance relations between them.  For that
reason  we start in \S\ref{sec:3} a theory of
\textbf{supertropical semirings} which are \textbf{totally
ordered}. The total order on such a semiring $U$  has to be
compatible with addition and multiplication, and has to extend the
order on $M = eU$ as a bipotent semiring.

 A supervaluation with values in a totally ordered semiring $U$
 will be called a \textbf{value-ordered supervaluation} or
 \textbf{\vo-supervaluation}, for short. Given two \vo-valuations $\vrp: R \to
 U$ and
 $\psi: R \to U$ we establish in \S\ref{sec:5} a refined notion of
 dominance, called \textbf{total dominance} and written
 $\geq_{\tot}$, which is sharper than the dominance relation $\vrp \geq
 \psi$ considered in \cite{IKR1}  and \cite{IKR2}. If $\vrp$ is
 surjective it means that the transmission $\al : U \to V$  with
 $\al \circ \vrp = \psi$ respects the orderings of $U$ and $V$. We
 then say that the transmission $\al$ is \textbf{monotone.}

 All examples of supervaluations in \S\ref{sec:3}-\S\ref{sec:6}
 will be \vo-supervaluations and all discussed  transmissions between
 them will be monotone.

It seems desirable to have a theory of \vo-supervaluations and
monotone transmissions at hand which parallels the theory of
supervaluations and transmissions in  \cite{IKR1}  and
\cite{IKR2}. The present paper only takes first steps in such a
theory, just enough to obtain a rich  stock of examples of
\vo-supervaluations and transmissions. An advantage of the
examples is that the total orderings ease the insight into the
structure of such supervaluations and transmissions compared to
cases where total orderings are not present or not respected.

An important point here is that  every monotone transmission  is a
semiring homomorphism (cf. Theorem \ref{thm5.3} below), while --
as we known from \cite{IKR1}  and \cite{IKR2} -- there exist many
transmissions which are not homomorphisms. Thus the examples do
not reflect certain aspects of general supervaluation theory.

A full fledged  theory of \vo-supervaluations should embrace an
analysis of the \vo-superval-uations  $\vrp: R \to U$ on a ring
$R$ equipped with  a cone or prime cone $T$ (cf. e.g.
\cite[Definitions~4.2.1, 4.3.1]{BCR}) which are compatible with
$T$ and the total ordering of $U$ in an appropriate sense. It
should have relevance for real algebraic geometry. We have to
leave these matters for future investigation.

\begin{notations*}
Given sets $X,Y$ we mean by $Y \subset X$ that $Y$ is a subset of
$X$, with $Y  = X$ allowed. If $E$ is an equivalence relation on
$X$ then $X/E$ denotes the set of $E$-equivalence classes in $X$,
and $\pi_E: X \to X/E$ is the map which sends an element $x$ of
$X$ to its $E$-equivalence class, which we denote by $[x]_E$. If
$Y \subset X$, we put $Y/E := \{[x]_E  \ds | x \in Y\}.$


$\tT(U)$ and $\tG(U)$ denote the sets of tangible and ghost
elements of $U$, respectively, cf. \cite[Terminology 3.7]{IKR1}.

If $v : R \to M $ is an \m-valuation we call the ideal $v^{-1}(0)$
of $R$ the \textbf{support} of $v$, and denote it by $\supp(v)$.
\end{notations*}

\section{$V$-valuations and $\Vz$-valuations}\label{sec:1}

Given any $m$-valuation $v: R\to M$ on a semiring $R$, we
introduce the sets
$$A_v:=\{x\in R \ds |v(x)\le 1\},$$
$$\mfp _v:=\{x\in R \ds |v(x)<1\}.$$
Clearly, $A_v$ is a sub-semiring of $R$ and~$\mfp _v$ is a prime
ideal of $A_v.$ Moreover, the sets $R\setminus \mfp _v$ and~
$R\setminus A_v$ are both closed under multiplication. The set
$R\setminus A_v$ may be empty, but $R\setminus \mfp _v$ is not,
since $1\notin \mfp _v.$

\begin{defn}\label{defn1.1}
We call $A_v$ the \bfem{valuation semiring} of $v,$ and $\mfp_v$
the \bfem{valuation ideal} of~$v.$
\end{defn}

If $R=F$ happens to be a semifield and $v: F\to M$ is a surjective
$m$-valuation, hence a surjective valuation (cf.~our terminology
in \cite[\S2]{IKR1}), then we meet a situation very similar to the
classical case that $F$ is a field. Now $A:=A_v$ has the property
that for any $x\in F^*$ either $x$ or $x^{-1}$ is an element of
$A,$ and for $x,y\in F^*$
$$v(x)\le v(y) \ \Leftrightarrow \ \frac{x}{y}\in A^* \ \Leftrightarrow \ xA
\subset yA.$$ Thus, the valuation $v$ is determined up to
equivalence by the sub-semiring $A_v=A$ of $F.$ It is also
uniquely determined by the set $\mfp _v$ since
$$A =\{z\in F \ds |z\mfp _v\subset \mfp _v\}.$$

Notice that $M$ is now a bipotent semifield,
$M=\Gamma\dot\cup\{0\}$ with $\Gamma$ an ordered abelian group.
This group can be identified with the group $F^*/A^*,$ since for
$x,y\in F^*$ we have
$$v(x)=v(y)\ \Leftrightarrow \ xA^*=yA^*.$$
\{$A^*$ denotes the group of units of the semiring $A.$\} Thus, we
may also write $M\cong F/A^*$, i.e., $M$ is the quotient of the
semifield~$F$ by the orbital equivalence relation on~$F$ given by
$A^*.$

In the case that $v$ is strong (which is automatic if $F$ is a
field), even the subgroup $A^*$ of~$F$ determines $v$ up to
equivalence. Indeed, now
$$A=A^*\cup\{x\in R \ds |1+x\in A^*\}.$$


In general,  matters are much more complicated. In the present
section our first goal is to coarsen a given surjective
$m$-valuation $v: R\twoheadrightarrow N$ ``slightly" in such a way
that the
 $m$-valuation $w:
R\twoheadrightarrow N$ has the same valuation ideal $\mfp_w=\mfp
_v$ as $v$ (as a subset of $R$ closed under addition and
multiplication), but $w$ is determined by the set $\mfp=\mfp _w$
in a canonical way. \{$w$ is a so-called ``$\Vz$-valuation", see
below.\} We then will pursue the same program based on the set
$A_v$ instead of $\mfp _v.$


\vskip 10pt

 In the following $R$
is always a (commutative) semiring.

\begin{defn}[cf. {\cite[\S 1]{C}}]\label{defn1.2}
Let $\mfp$ be a subset of $R$ with $$0\in \mfp, \quad 1\notin
\mfp, \quad \mfp+\mfp \subset \mfp,$$ and both $\mfp$ and $R
\setminus \mfp$ closed under multiplication. Then we call $\mfp$ a
\bfem{prime} of $R$.
\end{defn}

\begin{example}\label{examp1.3}
If $v: R\to M$ is any $m$-valuation, then $\mfp_v$ is a {prime} of
$R$.\end{example}

Let $\mfp$ be a {prime} of $R$. For any $x\in R$ and subset $L$ of
$R$, we put
$$[L:x]=\{z\in R \ds |zx\in L\}.$$
We define on $R$ an equivalence relation $\sim$ as follows:
$$x\sim y \ds \rightleftharpoons[\mfp :x]=[\mfp :y].$$

We observe that this equivalence relation is multiplicative, i.e.,
$x\sim y$ implies $xz\sim yz$ for any $z\in R.$ Indeed, suppose we
have a triple $x,y,z$ with $x\sim y$, but $xz\not\sim yz,$ say,
$[\mfp :xz]\not\subset [\mfp :yz].$ Then there exists some $u\in
R$ with $uyz\in \mfp ,$ but $uxz\notin \mfp.$ Thus $uz\in[\mfp
:y],$ but $uz\notin [\mfp :x].$ This contradicts the equality
$[\mfp:x]=[\mfp :y].$

We introduce the monoid
$$M:=M (R,\mfp):=(R/\sim,\cdot \; )$$
with the multiplication
$$[x]\cdot[y]:=[xy],$$
where $[x]$ denotes the equivalence class of $x$. Clearly, we have
a partial ordering $\le$ on the set $M$ given by
$$[y]\le[x] \dss \rightleftharpoons[\mfp :x]\subset [\mfp :y].$$
We start out to prove that this partial ordering is in fact total.

We will use the following lemma, now for the set $L = \mfp$, but
later also in other  situations.
 Let $L$ be a subset of $R$ such that both $L$ and $R \setminus L$
 are closed under multiplication.
\begin{lem}\label{lem1.4}
Let $x,y, s ,t\in R,$ and assume that $sx\in L $ and $ty\in L .$
Then at least one of the elements $sy,tx$ lies in $L$.\end{lem}

\begin{proof}
Since $L $ is closed under multiplication, we have
$$sy\cdot tx=sx\cdot ty\in L .$$
Since $R\setminus L $ is closed under multiplication, we conclude
that $sy\in L $ or $tx\in L .$\end{proof}

\begin{prop}\label{prop1.5}
Let $x,y\in R$ and $[\mfp :x]\not\subset[\mfp :y].$ Then $[\mfp
:y]\subset [\mfp :x].$
\end{prop}

\begin{proof}
We pick some $z\in R$ with $zx\in \mfp ,$ but $zy\notin \mfp.$ Let
$u\in[\mfp :y]$ be given. We have $zx\in \mfp ,$ $yu\in \mfp ,$
but $zy\notin \mfp .$ We conclude by the lemma that $ux\in \mfp ,$
i.e., $u\in [\mfp :x].$ This proves the claim.\end{proof}

Thus, the ordering on $M$ is total. The equivalence class $[0]$ is
the smallest element of $M$, since
$$[\mfp :0]=R\supset[\mfp :x]$$
for every $x\in R.$ Observe also that our ordering is compatible
with the multiplication on the monoid $M.$ Indeed,
$[\mfp:x]\subset[\mfp :y]$ implies $[\mfp :xz]\subset[\mfp :yz]$
for every $z\in R,$ as is easily seen.

We regard $M$ as a bipotent semiring, defining the addition on $M$
in the usual way (cf.~\cite[\S1]{IKR1}): $$ \text{If } \ [x]\le
[y], \text{ then } \ [x]+[y] :=[y].$$ We have $[0]=0_M,$
$[1]=1_M.$

\begin{thm}\label{thm1.6}
{}\quad

\begin{enumerate}
\item[a)] The map
$$v=v_{R,\mfp
}: R\to M,\qquad v(x):=[x],$$ is an $m$-valuation on the semiring
$R.$

\item[b)] Its support is
$$\mfq  :=\{x\in R \ds |Rx\subset \mfp \},$$
and this is a prime ideal of $R.$\end{enumerate}
\end{thm}

\begin{proof}
a): Clearly $v(0)=0,$ $v(1)=1,$ and $v(xy)=v(x)v(y)$ for any
$x,y\in R.$

It remains to verify for any $x,y\in R$  with $v(x) \leq v(y)$
that $v(x+y)\le v(y),$ i.e., $[\mfp :y]\ds \subset $ $[\mfp
:x+y].$ Given $z\in [\mfp :y],$ we have $zy\in \mfp $. This
implies $zx\in \mfp$ and then $z(x+y)\in \mfp.$ \{N.B. Here we use
for the first time that $\mfp + \mfp \subset \mfp .$\} Thus $[\mfp
:y]\subset [\mfp :x+y],$ i.e., $v(x+y)\le v(y),$ as desired.
\pSkip

 b): Given $x\in R,$ we have $v(x)=0$ iff $x\sim
0,$ i.e., $[\mfp ;x]=[\mfp :0]=R.$ Thus, $v$ has the support
$$\mfq :=\{x\in R \ds|Rx\subset \mfp \}.$$  Now, if $x,y\in
R\setminus\mfq ,$ there exist elements $s,t\in R$ with $sx\notin
\mfp ,$ $ty\notin \mfp .$ It follows that $stxy\notin \mfp ,$ and
hence $xy\notin\mfq .$ Thus, $\mfq  $ is a prime ideal of $R.$
\end{proof}

\begin{defn}\label{defn1.7}
We say that $$v_{R,\mfp}: R \ds \to M(R,\mfp)$$ is the
$m$-\bfem{valuation associated to the  prime $\mfp$ of $R$}. We
call any $m$-valuation equivalent to such a valuation $v_{R,\mfp}$
a ${\bf \Vz}$-\bfem{valuation}. Later (from \S\ref{sec:3} onward),
we often write $v_{\mfp}$ instead of $v_{R,\mfp}.$
\end{defn}

The construction of these $m$-valuations is in some sense dual to
the construction of the ``$V$-valuations" in the paper \cite{HV1}
by Harrison and Vitulli; hence the label $\Vz$. We will discuss
$V$-valuations below.

We compute the valuation semiring and valuation ideal of a
$\Vz$-valuation $v_{R,\mfp}.$

\begin{prop}\label{prop1.8}
Let $v=v_{R,\mfp}$ for a prime $\mfp$ of $R$. Then $A_v=\{x\in R
\ds | x \mfp \subset \mfp \}$ and $\mfp _v=\mfp .$
\end{prop}

\begin{proof}
Let $x\in R.$

\noindent a) $x\in A_v\Leftrightarrow v(x)\le 1\Leftrightarrow
[\mfp:x]\supset[\mfp :1]=\mfp \Leftrightarrow x\mfp \subset \mfp.$
\pSkip

\noindent b) $x\in \mfp _v\Leftrightarrow v(x)< 1\Leftrightarrow
[\mfp:x] \supsetneqq[\mfp :1]=\mfp \Leftrightarrow \mfp x\subset
\mfp ,$ but there exists also some $s\in R\setminus \mfp $ with
$sx\in \mfp .$ Since both $\mfp$ and $R\setminus \mfp $ are closed
under multiplication, the last  condition means that $x\in \mfp.$
We conclude that $\mfp_v = \mfp$.
\end{proof}

\begin{lem}\label{lem1.9}
Let $v: R\twoheadrightarrow  M$ be a surjective $\Vz$-valuation.
Then the bipotent semiring $M$ has the following separation
property:
 \begin{align*}& \text{\emph{($\Sep^0_V$)}}: \ \text{If}\ \al ,\bt \in M \ \text{and}\
 \al <\bt ,\ \text{there exists some} \  \gm \in
 M\\   & \qquad \qquad \text{with}\  \al \gm <1\ \text{and}\
\bt \gm \ge 1.
\end{align*}

%
\end{lem}

\begin{proof}
Choose $x,y\in R$ with  $v(x)=\al ,$ $v(y)=\bt .$ Then $[\mfp
:x]\varsupsetneqq[\mfp:  y].$ Thus, there exists some $z\in R$
with $zx\in\mfp $ but $zy\notin \mfp .$ This means that $v(zx)<1$
but $v(zy)\ge1.$ The element $\gm :=v(z)$ does the job.\end{proof}

\begin{defn}\label{defn1.10}
We call a bipotent semiring $M$ having the separation property
 {\rm (Sep}$^0_V${\rm)} a
\bfem{${\bf \Vz}$-semiring}.\end{defn}

We now can state a remarkable fact.

\begin{thm}\label{thm1.11}
Assume that $R$ is a semiring, $M$ is a $\Vz$-semiring, and
$v:R\to M$ is  a surjective map with $v(0)=0,$ $v(1)=1$,
$v(xy)=v(x)v(y)$ for any $x,y\in R$ (i.e., $v$ is a homomorphism
from the monoid $(R,\cdot \; )$ onto the monoid $(M,\cdot \; )).$
Assume also that
\begin{equation}\label{1.1}
\forall x,y\in R:\quad  v(x)<1,v(y)<1\ \Rightarrow\ v(x+y)<1.
\end{equation}
Let $\mfp :=\{x\in R \ds |v(x)<1\}.$ Then $\mfp$ is a prime of
$R$, and $v$ is a $\Vz$-valuation equivalent to~$v_{R,\mfp }.$
\end{thm}

\begin{proof}
It is obvious that $0\in \mfp ,$ $1\notin\mfp $ and both $\mfp$
and $R \setminus  \mfp $ are closed under multiplication. The rule
\eqref{1.1} tells us that $\mfp +\mfp \subset \mfp .$ Thus $\mfp $
is a prime of $R$. We will verify that
\begin{equation}\label{1.2}
\forall x,y\in R:\quad  v(x)\le v(y) \ds \Leftrightarrow\ [\mfp :
x]\supset[\mfp :y].\end{equation} Then we will be done. Indeed,
let $w:=v_{R,\mfp },$ $N:=M(R,\mfp).$ We know by \eqref{1.2} that
for any $x,y\in R,$ $v(x)=v(y)$ iff $w(x)=w(y).$ Thus we have a
well-defined bijection $\gm : M\to N$ with $\gm (v(x))=w(x)$ for
all $x\in R.$ This map sends 0 to 0,\ 1 to 0, and is
multiplicative. Further, \eqref{1.2} tells us that $\gm $ is order
preserving. Thus $\gm $ is a semiring isomorphism and $\gm \circ
v=w.$

Instead of \eqref{1.2} we verify the equivalent property
$$\forall x,y\in R: \quad v(x)>v(y)\ \Leftrightarrow\ [\mfp:x]\varsubsetneqq[\mfp :y].$$
If $v(x)>v(y)$ then, of course, for every $z$ with $v(xz)<1, $
i.e., $xz\in \mfp ,$ we have $v(yz)< 1$ i.e., $yz\in \mfp .$ But
since $M$ is $\Vz$, there exists some $z'\in R$ with $v(xz')\ge1,$
$v(yz')<1, $ i.e., $xz'\notin \mfp ,$ $yz'\in \mfp .$ Thus
$[\mfp:x]\varsubsetneqq[\mfp :y].$ On the other hand, if
$[\mfp:x]\varsubsetneqq[\mfp :y],$ we have some $z\in R$ with
$v(xz) \geq 1,$ $v(yz) <  1,$ and hence $v(xz) >  v(yz).$ Thus
certainly $v(x) > v(y).$
\end{proof}

\begin{thm}\label{thm1.12}
Assume that $v: R\to M$ is a surjective $m$-valuation. Let
$$w:=v_{R,\mfp }:R \ds \to N:=M(R,\mfp)$$ denote the $\Vz$-valuation
associated to the prime $\mfp := \mfp _v.$

Then $v$ dominates $w,$ i.e., there exists a (unique) semiring
homomorphism $\gm : M\twoheadrightarrow N$ such that $w=\gm \circ
v.$ In other terms, $w$ is a coarsening of $v,$
cf.~\cite[\S2]{IKR1}, \cite{IKR2}.
\end{thm}

\begin{proof} We only need to verify that, for any $x,y\in R,$
$v(x) \leq v(y)$ implies $w(x) \leq w(y),$
cf.~\cite[Definition~2.9]{IKR1}. But this is obvious. If $v(x)
\leq v(y),$ then
$$[\mfp :x]=\{z\in R \ds | v(zx) \in  \mfp \}\supset\{z\in R \ds |v(zy)\in  \mfp
\}=[\mfp:y].$$\end{proof}

\begin{defn}\label{defn1.13} We call $w$ the
\bfem{$\bf \Vz$-coarsening} of $v,$ and we write $w=\darv.$
\end{defn}

\begin{example}\label{examp1.14}
Let $M$ be any bipotent semiring. The identity map $v=\id_M$ may
be viewed as a (strict) $m$-valuation on the semiring $M.$ Thus it
gives us a $\Vz$-valuation $$\gm_V^0 := \gm^0_{M,V} :=
(\id_M)_\downarrow: M\twoheadrightarrow \darM.$$ Since $v$ is
strict, its coarsening $\gm_V^0$ is again strict, i.e., $\gm_V^0$
is a semiring homomorphism. The associated homomorphic equivalence
relation is given by
$$x\sim y\ds \Leftrightarrow [\mfp _M:x]=[\mfp _M:y]$$
with $$\mfp _M:=\{x\in M \ds|x<1\}.$$ The valuation semiring of
$\gm_V^0$ is
$$[ \mfp_M: \mfp_M ] :=\{x\in M\ds |x \mfp_M \subset   \mfp_M\},$$ and the valuation ideal of $\gm_V^0$ is
$\mfp _M.$\end{example}

We retain the notations developed in this example. The map
$\gm_V^0$ allows us a fresh view of the $\Vz$-coarsening of any
surjective $m$-valuation $v: R \twoheadrightarrow M.$

\begin{prop}\label{prop1.15}
The $\Vz$-coarsening $\darv$ of $v$ is equivalent to the
$m$-valuation
$$\gm_V^0\circ v: R\to\darM.$$
\end{prop}

\begin{proof} Let $\mfp :=\mfp _v.$ We have
${\gm_V^0}^{-1}(\mfp _{\darM})=\mfp _M$ and $v^{-1}(\mfp_M)=\mfp
.$ Thus, $w:=\gm_V^0\circ v: R\to \darM$ has the valuation ideal
$\mfp _w=\mfp .$ Moreover, $\darM$ is a $\Vz$-semiring.
\thmref{thm1.11} gives the claim.\end{proof}

We now turn to the construction of $V$-valuations, to be found in
\cite{HV1} (in the  case that $R$ is a ring). As before $R$ may be
any semiring.

\begin{defn}\label{defn1.16} \cite{HV1}. Let $A$ be a subset of $R$
with $$ 0 \in A, \quad 1 \in A, \quad A+A \subset A,$$ and both
$A$ and $R \setminus A$   closed under multiplication.  Then we
call $A$ a \textbf{CMC-subsemiring of} $R$.
\end{defn}

In other words, a set $A \subset R$ is a CMC-subsemiring of $R$
iff $A$ is a subsemiring of $R$ and~ $R \setminus A$ is closed
under multiplication. The label CMC (= complement multiplicatively
closed) alludes to this latter property. Notice that, if $R$
happens to be a ring, then the relation $(-1) \cdot (-1) = 1$
forces $(-1)$ to be in $A$, hence $A$ is a subring of $R$.

Let now $A$ be a CMC-subsemiring of $R$, which is \textbf{proper},
i.e. $A \neq R$. Then, in complete analogy to the above, we define
an equivalence relation $\sim$ on $R$ by
$$ x \sim y \rightleftharpoons [A:x] = [A:y].$$
This equivalence relation is again multiplicative; hence we obtain
a monoid
$$ M := M(R,A) := (R/\sim, \cdot \ )$$
with the multiplication
$$ [x] \cdot [y] := [xy],$$
and we can see as above (in particular use Lemma \ref{lem1.4} for
$L =A$) that this monoid $M$ is totally ordered by the rule
$$ [y] \leq [x] \ \Leftrightarrow \ [A:x] \subset [A:y]. $$
One further verifies (cf. \cite{HV1}) that the map
$$
v:= v_{ A}: =v_{R,A}:R \ds \to M $$ is an $m$-valuation on the
semiring $A$ with support
$$\mfq := \{ x\in R \ds | Rx \subset A\} $$
which  again is a prime ideal of $R$.

\begin{defn}\label{defn1.17} \cite{HV1}.We call any $m$-valuation
$v$ on $R$ which is equivalent to $v_{R,A} $  for some proper
CMC-subsemiring $A$ of $R$ a \textbf{$\bf V$-valuation} on $R$
and we call $v_{R,A} $ the \textbf{$\bf V$-valuation of $R$
associated to $A$.}
\end{defn}

  \emph{Historical comments:} In \cite{HV1}
these valuations have been dubbed ``finite $V$-valuations" by
Harrison, to give credit to Marie A. Vitulli, who conveyed to him
the idea of this construction and in addition a construction of
``infinite $V$-valuations", a type of absolute values having an
archimedian flavour. We will not deal with infinite $V$-valuations
here; hence we simply speak of finite $V$-valuations as
``$V$-valuations".

Harrison and Vitulli report that already M. Griffin defined
(finite) $V$-valuations in an unpublished paper \cite{Gr}, but
then did not pursue this idea further \cite[p.~269]{HV1}.

$\Vz$-valuations have been introduced - in the case of rings -  by
D. Zhang \cite{Z}. \hfill\quad \qed

\pSkip

 If $v = v_{R,A}$ for some proper CMC-subsemiring $A$ of
$R$, then it easily checked  that $v$ has the valuation semiring
$A_v = A$. Further it turns out that the valuation ideal $\mfp_v$
of $v$ is the set  \begin{equation}\label{eq:1.3-0} P(A) := \{
x\in R \ds | \exists y \in R \setminus A : xy \in A \}.
\end{equation}
In particular $P(A)$ is a prime of $R$, hence a prime ideal of
$A$, a fact which for $R$ a ring had been already been observed by
P. Samuel in his seminal paper \cite{S}. \{Samuel's direct proof
also  goes through verbatim  for $A$ a semiring.\}

\begin{defn}\label{defn1.18} We call $P(A)$ the \textbf{central
prime} (in  $R$) of the proper CMC-subsemiring~A.
\end{defn}

Parallel to Lemma \ref{lem1.9} we obtain the following:

\begin{lem}\label{lem1.19} \cite{HV1}.
Let $v: R\twoheadrightarrow  M$ be a surjective $V$-valuation.
Then the bipotent semiring $M$ has the following separation
property:
 \begin{align*}& \text{\emph{($\Sep_V$)}}: \ \text{If}\ \al ,\bt \in M \ \text{and}\
 \al <\bt ,\ \text{there exists some} \  \gm \in
 M\\   & \qquad \qquad \text{with}\  \al \gm  \leq 1\ \text{and}\
\bt \gm  >  1.
\end{align*}
\end{lem}

\begin{defn}\label{defn1.20} For any bipotent semiring $M$ we
introduce the CMC-subsemiring
$$ A_M := \{ x\in M \ds | x \leq 1\},$$
and we call $M$ a \textbf{proper bipotent semiring} if $A_M \neq
M$. \{N.B. If $M$ is cancellative this means that $M$ is
``unbounded", i.e., does not have a largest element.\} We call the
bipotent semiring $M$ a \textbf{$\bf V$-semiring} if $M$ is proper
and has the separation property ($\Sep_V$).
\end{defn}

We now  can deduce results parallel to Theorem \ref{thm1.11} --
Proposition \ref{prop1.15} by arguing in exactly the same way as
above. We first obtain

\begin{thm}\label{thm1.21}
Let  $R$ be a semiring, $M$ a $V$-semiring, and $v:R\to M$  a
surjective map with $v(0)=0,$ $v(1)=1$, $v(xy)=v(x)v(y)$ for all
$x,y\in R$ (i.e., $v$ is a homomorphism from the monoid $(R,\cdot
\; )$ onto the monoid $(M,\cdot \; )).$ Assume  that
\begin{equation}\label{1.3}
\forall x,y\in R: v(x)\leq 1,v(y) \leq1\ \Rightarrow\ v(x+y)\leq
1.
\end{equation}
Let $A :=\{x\in R|v(x) \leq 1\}.$ Then $A$ is a proper
CMC-subsemiring of $R$, and $v$ is a $V$-valuation equivalent
to~$v_{R,A }.$
\end{thm}

Then, given any surjective $m$-valuation $v: R \twoheadrightarrow
M$ with $M$ proper, we obtain the  $\bf V$-\textbf{coarsening}
$$ \uarv: R \ds \twoheadrightarrow \uarM$$
of $v$, which is the finest coarsening of $v$ to a $V$-valuation.
This  is associated to the CMC-subsemiring $A := A_v$ of $R$.

Starting with a proper bipotent semiring $M$ we may apply this to
$v = \id_M$ and then get a surjective homomorphism
$$\gm_V := \gm_{M,V} := (\id_M)^\uparrow:
M\twoheadrightarrow \uarM,$$ with $\uarM$ a $V$-semiring. We have
\begin{equation}\label{eq:1.4} \gm_V(x) \leq \gm_V(y) \quad
\text{iff} \quad [A_M:x] \supset [A_M:y].
\end{equation}
Finally we observe that the $V$-coarsening of any surjective
$m$-valuation $v:R \to M $ with $M$ proper is given by
\begin{equation}\label{eq:1.5} \uarv = \gm_{M,V} \circ v.
\end{equation}

\section{Turning $m$-valuations into valuations}\label{sec:2}

In  \S\ref{sec:1} we have seen classes of surjective
$m$-valuations $v:R\to M,$ where the bipotent semiring $M$ in
general has no reason to be cancellative, i.e., $v$ is not a
valuation. The problem arises how to handle such true
$m$-valuations in a supertropical context.

Our way to do this is to find a coarsening $w: R \onto N$ of $v$,
as ``slight" as possible, which is a valuation and has the same
support as $v$, $w^{-1} (0)  = v^{-1} (0) $. Then, as explained
below in~\S\ref{sec:3}, we can interpret  $v$ as a tangible
supervaluation covering $w$.

Given any bipotent  semiring $M$, we look for a homomorphic
equivalence relation  $C$ on $M$, as fine as possible, such that
the semiring $M/C$ is cancellative. If it happens that~$\{ 0 \}$
is a $C$-equivalence class in $M$, then for any surjective
\m-valuation $v:R \onto M$ we will have the valuation
$$ w := v/C := \pi_C \circ v : R \onto M/C$$
at our disposal, and  $w^{-1} (0)  = v^{-1} (0) $, as desired.

 We always assume that $M$ is different from the  zero ring $\{0\}.$ Let $\mfq$ denote
the nilradical of $M,$ i.e.,
$$\mfq:=\Nil M=\{x\in M\ds|\exists n \in\N:x^n=0\}=\sqrt{\{0\}}.$$
%
\begin{lem} $\mfq$ is a lower set and a prime ideal of $M$.

\end{lem}

\begin{proof} a) Assume that $x \leq y$ and $y^n = 0$. Clearly $y
\leq 1$, hence $y^n \leq 1$ for all $n \in \N$. Thus
$$ 0 = y^n \geq y^{n-1} x \geq \dots \geq y x^{n-1} \geq x^n,$$
and we conclude that $x^n =0$. Thus $\mfq$ is a lower set of $M$.
\pSkip

b) Clearly $y \cdot M \subset \mfq$ for any  $y \in \mfq $. Also
$1 \notin \mfq$. Thus $\mfq$ is a proper ideal of $M$. \pSkip

c) Let $x,y \in M$ be given with $xy \in \mfq$, and assume that $x
\leq y$. We have $x < 1$. Indeed, $x \geq 1$ would imply $xy \geq
y \geq 1$, but the lower set $\mfq$ does not contain $1$. It
follows that $x^2 \leq xy$. Since $\mfq$ is a lower set, we infer
that $x^2 \in \mfq$, hence $x \in \mfq$. Thus the ideal $\mfq$ is
prime.
\end{proof}

Notice that $\mfq=\{0\}$ iff $M\setminus\{0\}$ is closed under
multiplication, i.e., $M$ is a semidomain. \pSkip

Our ansatz for C is the following binary relation on the set $M$.
\begin{equation} x\sim_Cy\Leftrightarrow \left\{ \begin{array}{l}
  \text{Either\ }                         x,y\in\mfq, \\[1mm]
                                           \text{or there exists some}\ s\in
M\setminus\mfq\ \text{with}\ sx=sy. \\
                                         \end{array}\right.
\end{equation}

We verify that this is an equivalence relation on the set $M.$
Only transitivity needs a proof. Let $x,y,z\in M$ be given with
$x\sim_Cy$ and $y\sim_Cz.$ If at least one of the elements $x,y,z$
lies in $\mfq$, then all are  in $\mfq.$ Otherwise we have
elements $s,t$ in $M\setminus\mfq$ with $sx=sy$ and $ty=tz.$ This
implies $stx=stz$ and $st\in M \setminus \mfq.$ Thus $x\sim_C z$
in both cases.

\begin{thm}\label{thm2.1}\quad {}

\begin{enumerate}
\item[a)] $C$ is a homomorphic equivalence relation on $M.$ Thus
we have a unique structure of a (bipotent) semiring on $M/C$ such
that the natural map
$$\pi_C: M\to M/C,\quad x\mapsto [x]_C$$
is a homomorphism.

\item[b)] $M/C$ is cancellative and $\pi_C^{-1}(0)=\mfq
:=\Nil M.$

\item[c)] If $\gm : M\to N$ is a homomorphism from $M$ to a
cancellative semiring $N$ with $\gm ^{-1}(0)=\mfq,$ then $\gm $
factors through $\pi_C$ in a unique way.
\end{enumerate}
\end{thm}

\begin{proof} a): Let $x,y,z\in M$ be given with $x\sim_C y.$ It
is fairly obvious that $xz\sim_Cyz.$ We have to verify that also
$x+z\sim_Cy+z.$

\begin{description}
\item[Case 1] $x,y\in\mfq.$ If $z\in\mfq,$ then $x+z,y+z\in\mfq.$
If $z\notin \mfq,$ then $x<z$ and $y<z$, hence $x+z=z=y+z.$

\item[Case 2] $x,y\notin\mfq.$ There exists some $s\in
M\setminus\mfq$ with $sx=sy.$ It follows that $s(x+z)=s(y+z).$
\end{description}

Thus $x+z\sim_C y+z$ in both cases. We have verified that the
equivalence relation $C$ is homomorphic. \pSkip

b): By definition $x\sim_C 0$ iff $x\in\mfq.$ Thus the
homomorphism $\pi_C:M\twoheadrightarrow M/C$ has the kernel
$\pi_C^{-1}(0)=\mfq.$ We now verify that $M/C$ is cancellative.
Let $x,y,z\in M$ be given with $[z]_C\ne0$ and
$[x]_C\cdot[z]_C=[y]_C\cdot[z]_C.$ In other words, $z\in
M\setminus \mfq$ and $xz\sim_C yz.$ We want to prove that
$[x]_C=[y]_C,$ i.e., $x\sim_Cy.$

If $x\in\mfq,$ we have $xz\in\mfq,$ hence $yz\in\mfq.$ Since
$z\notin\mfq$, this implies $y\in\mfq.$ Assume now that $x,y\notin
\mfq.$ Then $xz,yz\notin \mfq.$ Since $xz\sim_Cyz$ there exists
some $s\in M\setminus \mfq$ with $xzs=yzc.$ We have $zs\in
M\setminus\mfq,$ hence $x\sim_C y$ again. \pSkip

c): Let $\gm : M\to N$ be a homomorphism with $N$ a cancellative
semiring and $\gm ^{-1}(0)=\mfq.$ Given $x,y\in M$ with $x\sim_C
y$ we want to verify that $\gm (x)=\gm (y).$

\begin{description}
\item[Case 1] $x\in\mfq.$ Then $y\in\mfq.$ Since $x,y$ are nilpotent
and $N$ is a semidomain, we conclude that $\gm (x)=0=\gm (y).$
\{N.B. Here we did not yet need the hypothesis that $\gm
^{-1}(0)=\mfq.$\} \pSkip

\item[Case 2] $x\notin\mfq.$ Now $y\notin\mfq,$ and there exists
some $s\in M\setminus \mfq$ with $sx=sy.$ It follows that $\gm
(s)\gm (x)=\gm (s)\gm (y).$ Since $\gm ^{-1}(0)=\mfq,$ we have
$\gm (s)\neq 0.$ Since $N$ is cancellative, we conclude that $\gm
(x)=\gm (y)$ again. \end{description}

Thus $\gm $ induces a well-defined map $\bar\gm : M/C\to N,$ given
by
$$\bar\gm ([x]_C):=\gm (x)$$
for all $x\in M.$ This is a homomorphism, and $\bar\gm
\circ\pi_C=\gm .$ Since $\pi_C$ is surjective, we have no other
choice for $\bar\gm .$
\end{proof}

\begin{notations} $ $
\begin{enumerate}\eroman
    \item  We call $C$ the \bfem{minimal cancellative
relation} on $M.$ If necessary we more precisely write $C(M)$
instead of $C.$
    \item  If $v: R\to M$ is an $m$-valuation on a semiring $R,$ we
denote the coarsening
$$\pi_C\circ v: R\to M\to M/C$$
by $v/C.$ It is a valuation.
\end{enumerate}
\end{notations}

If $v$ is a $V^0$-valuation or a $V$-valuation, the question
arises whether $v/C$ is again $V^0$ or $V.$ In order to attack
this problem, it will be helpful to introduce two more classes of
$m$-valuations.

\begin{defn}\label{def2.3}  $ $
\begin{enumerate}\ealph
\item We say that an $m$-valuation $v: R\to M$ on a semiring $R$
has \bfem{unit incapsulation} (abbreviated \bfem{UIC}), if for any
$x,y\in R$ with $v(x)<v(y)$ there exists some $z\in R$ with
$$v(xz)\le 1\le v(yz).$$

\item We say that $v$ has \bfem{strict UIC}, if for any $x,y\in R$
with $v(x)<v(y)$ there exists some $z\in R$ with
$$v(xz)<1<v(yz).$$

\item  We say that a bipotent semiring $M$ has UIC (resp. strict
UIC), if the $m$-valuation $\id_M: M\to M$ has UIC (resp. strict
UIC). In other words, if for any $\al <\bt $ in $M$ there exists
some $\gm \in M$ with $\al \gm \le 1\le \bt \gm $ (resp. $\al \gm
<1<\bt \gm ).$

\end{enumerate}
\end{defn}

We have the chart of implications

$$\xymatrix{
\text{strict}\ UIC\ar @{=>}[r] & V\wedge V^0\ar@{=>}[d] \ar
@{=>}[r] \ar
 @{=>}[r] & V^0\ar@{=>}[d]\\ & V
  \ar @{=>}[r]  & UIC \ .
 }$$

The class UIC is particularly useful for the following property
not shared by the other classes.

\begin{lemma}\label{lem2.4}
If $v: R\to M$ has UIC and $w$ is a coarsening of $v,$ then also
$w$ has UIC.\end{lemma}

\begin{proof} We may assume that $v$ is surjective. Then
$w=\gm \circ v$ with $\gm : M\to N$ a semiring homomorphism. If
$x,y$ are elements of $R$ with $w(x)<w(y),$ then also $v(x)<v(y),$
hence there exists some $z\in R$ with $v(xz)\le 1\le v(yz),$ and
this implies $w(xz)\le 1\le w(yz).$\end{proof}

\begin{lemma}\label{lem2.5}
Assume that $M$ is a bipotent semiring with UIC. Then there does
not exist a saturated ideal of $M$ different from $M$ and $\{0\}.$
In particular,  $\Nil M = \{ 0 \}$,  hence $M$ is a
semidomain.\end{lemma}

\begin{proof}
Suppose $\mfq$ is a saturated ideal of $M$ different from $\{0\}.$
We choose some $x\in\mfq$ with $x\ne 0.$ Applying UIC to the pair
$0<x,$ we obtain some $z\in M$ with $1\le xz.$ Since $xz\in\mfq$
and $\mfq$ is saturated, the relation $1+xz=xz$ implies that
$1\in\mfq,$ hence $\mfq=M.$ Applying this to $ \Nil M$, we see
that $\Nil M=\{0\}.$
\end{proof}

It follows that, if the bipotent semiring $M$ has $UIC$, we have
the following simple description of the cancellation relation
$C=C(M):$
\begin{equation}\label{2.2}
x\sim_C y \ds \Leftrightarrow \exists z\in M\quad\text{with}\quad
z\ne0\quad\text{and}\quad xz=yz.
\end{equation}

Thus we can state
\begin{prop}\label{prop2.6}
If $v:R\to M$ is an $m$-valuation with UIC, then $w:=v/C$ can be
characterized as follows: For any $x,y\in R$
$$w(x)\le w(y) \ds \Leftrightarrow \exists z\in R\quad{with}\quad
v(z)\ne0\quad{and}\quad v(xz)\le v(yz).
$$
\end{prop}

\begin{schol}\label{schol2.7}
Taking into account our explicit construction of $V^0$- and
$V$-valuations in \S\ref{sec:1}, we see the following: Let $R$ be
any semiring.
\begin{enumerate}
\item[i)] If $v$ is the $V^0$-valuation associated to a prime
$\mfp$ of $R$, then for any $x,y\in R$
\begin{equation}\label{2.3}
(v/C)(x)\le (v/C)(y) \ \Leftrightarrow \ \exists z\in R\
\text{with}\  Rz\not\subset\mfp\ \text{and}\
[\mfp:xz]\supset[\mfp:yz].\end{equation} \item[ii)] If $v$ is the
$V$-valuation associated to a proper CNC-subring $A$ of $R,$ then
for any $x,y\in R$
\begin{equation}\label{2.4}
(v/C)(x)\le (v/C)(y) \ \Leftrightarrow  \ \exists z\in R\
\text{with}\ Rz\not\subset A\ \text{and}\ [A:xz]\supset
[A:yz].\end{equation}
\end{enumerate}
\end{schol}

We want to exhibit good cases, in which $v/C$ is a
$V^0$-valuation, or a $V$-valuation or even has strict UIC. For
that reason we analyze under which additional assumption a
cancellative bipotent semiring $M$ has one of the properties
$V^0,$ $V,$ strict UIC. We will use a self-explanatory notation.
For example, $M_{>1}$ denotes the set $\{x\in M|x>1\}.$

\begin{thm}\label{thm2.8} Assume that $M$ is a cancellative
bipotent semiring with UIC and $M_{<1}\ne \{0\}.$
\begin{enumerate}
\item[i)] If $M_{<1}$ has a biggest element, then $M$ is a
$V^0$-semiring. \item[ii)] If $M_{>1}$ has a smallest element,
then $M$ is a $V$-semiring. \item[iii)] If $M_{>1}=\emptyset,$
then trivially $M$ is a $V^0$-semiring. \item[iv)] Otherwise $M$
has strict UIC.\end{enumerate}\end{thm}

\begin{proof}
Given a pair $\al <\bt $ in $M$, there exists some $\gm \in M$
with $\al \gm \le 1\le \bt \gm ,$ since $M$ has UIC. We need to
modify $\gm $ in the various cases to get strict inequality at the
appropriate places. Since $M$ is cancellative, we know in advance
that $\al \gm <\bt \gm .$

i): We want to find some $\delta\in M$ with $\al \delta<1\le \bt
\delta.$ If $\al \gm <1$ we are done with $\delta=\gm .$ Assume
now that $\al \gm =1$ and $M_{<1}$ has a biggest element $p_0>0.$
Multiplying by $p_0$ we obtain
$$\al \gm  p_0= p_0<\bt \gm  p_0.$$
We conclude that $1\le \bt \gm  p_0,$ hence
$$\al \gm  p_0<1\le \bt \gm  p_0.$$ \pSkip

ii): Now  clear by the ``dual" argument to the just given proof of
i). \pSkip

iii): obvious. \pSkip

iv): Assume that $M_{<1}\ne\{0\}$, $M_{>1}\ne\emptyset,$ and
neither $M_{<1}$ has a biggest element nor $M_{>1}$ has a smallest
element. We have $\al ,\bt ,\gm \in M$ with $\al \gm \le 1\le \bt
\gm $ and $\al \gm <\bt \gm .$ We want to find some $\delta\in M$
with $\al \delta<1<\bt \delta.$

Either $\al \gm <1$ or $\bt \gm >1.$ By symmetry it suffices to
study the case $\al \gm <1.$ \{The setting is not entirely
symmetric, since $\al $ can be zero while $\bt $ cannot be zero,
but this does not matter.\} Since $M_{<1}$ has no biggest element,
we find $u,v\in M$ with $\al \gm <u< v<1.$ Due to UIC we have some
$\eta\in M$ with $u\eta\le 1\le v\eta.$ Again using that $M$ is
cancellative, we obtain
$$\al \gm  \eta< u\eta\le 1\le v\eta<\eta\le
\bt \gm \eta.$$ The element $\delta=\gm \eta$ does the job.
\end{proof}

\begin{schol}\label{schol2.9}
Let $v: R\to M$ be a surjective $V^0$-valuation or $V$-valuation
on a semiring~$R.$ Then $v$ has UIC, hence also $$v/C:
R\twoheadrightarrow \olM:=M/C$$ has UIC. Discarding the degenerate
case that $\olM_{<1}=\{0\},$ we read off from \thmref{thm2.8},
that~$v/C$ is again a $V^0$- or $V$-valuation. But in the cases
that $\olM_{<1}$ has a biggest element or $\olM_{>1}$ has a
smallest element, it may happen that $v/C$ has the opposite type
$(V$ resp. $V^0$) to~$v.$
\end{schol}

\section{Totally ordered supertropical predomains and ultrametric supervaluations}\label{sec:3}

We will see that $m$-valuations, which are not necessarily
valuations, can be interpreted as tangible supervaluations with
values in ``totally ordered supertropical semirings". By this we
mean the following.

\begin{defn}\label{defn3.1}
Let $U$ be a supertropical semiring. A \bfem{total ordering} of
this semiring is a total ordering $\le$ of the set $U$ which is
compatible with addition and multiplication, i.e., for all~
$x,y\in U$
\begin{equation}\label{3.1}
x\le y \ \Rightarrow \ x+z\le y+z,
\end{equation}
\begin{equation}\label{3.2}
x\le y \ \Rightarrow \ x \cdot   z \leq y \cdot  z,
\end{equation}
and moreover satisfies
\begin{equation}\label{3.3}
0\le1
\end{equation}
(hence $0\le z$ for all $z\in U).$
\end{defn}

\{N.B. More generally, this definition can be formulated for any
semiring
 $R$ with
$R\setminus\{0\}$ closed under addition.\}

We know that every supertropical semiring $U$ is $ub$ (= upper
bound), i.e., carries a partial ordering $\le$, defined by
$$x\le y \dss \Leftrightarrow \exists z\in U: x+z=y,$$
\cite[Proposition 11.9]{IKR1}. It  again obeys the rules
\eqref{3.1}--\eqref{3.3}. In the following, we call this partial
ordering the \bfem{minimal ordering} of $U.$ The reason for this
terminology is that in this ordering any inequality $x\le y$ is a
formal consequence of the rules \eqref{3.1}--\eqref{3.3}.

Any total ordering of $U$ clearly refines the minimal ordering on
$U.$ It is further evident that the restriction of the minimal
ordering of $U$ to the subsemiring $M:=eU$ is the minimal ordering
of $M,$ which is total. Since a total ordering cannot be further
refined, it follows that any total ordering of $U$ restricted to
$M$ gives the minimal ordering on the bipotent semiring~$M.$

We now write down  easy observations which tell us that the whole
structure of a totally ordered supertropical semiring $U$ can be
understood in terms of the totally ordered monoid~$(U,\cdot \; )$
and the ghost map $\nu_U,$ regarded as a map from $U$ to $U$ (and
denoted here $p$).

\begin{schol}\label{schol3.2}
Assume that $U$ is a totally ordered supertropical semiring.
\begin{enumerate}
\item[a)] $(U,\cdot \; )$ is a totally ordered monoid with absorbing
element $0.$ \item[b)] The map\footnote{The map $p$ given here is
analogous to the ``ghost map'' $\nu$ given for semirings with
ghosts in \cite{IzhakianRowen2007SuperTropical}.} $p:U\to U,$
$x\mapsto ex$ has the following properties. \{The label ``$\Gh$"
alludes to ``ghost map".\}

\begin{align*}
& (\Gh1): \quad   p\circ p=p.  \\
& (\Gh2): \quad   p^{-1}(0)=\{0\}.  \\
& (\Gh3): \quad   \forall x,y\in U: p(xy)=p(x)p(y).  \\
& (\Gh4): \quad  \forall x,y\in U: x\le y\Rightarrow p(x)\le p(y).  \\
& (\Gh5): \quad   \forall x\in U: x\le p(x).  \\
\end{align*}
 \item[c)] The
addition of the semiring $U$ is determined by the ordering of $U$
and the  map $p$ as follows: For all $x,y \in  U$
$$x+y=\begin{cases} y &\quad \text{if}\ p(x)<p(y),\\
x&\quad \text{if}\ p(y)<p(x),\\
p(x) &\quad \text{if}\ p(x)=p(y).\end{cases}$$
 \item[d)]$M:=p(U)$ is an ideal of $U.$ If $U\setminus M$ is
 closed under multiplication, then the totally ordered monoid
 $(M\setminus\{0\},\cdot \; )$ is cancellative (cf. \cite[Proposition
 3.13]{IKR1}).
\end{enumerate}
\end{schol}
The  observations just made lead us
to a construction of all ordered supertropical predomains,  i.e.,
order supertropical semirings $U$ such that $U\setminus eU$ is not
empty and closed under multiplication and $eU$ is cancellative
\cite[Definition 3.14]{IKR1}. (The case $U=eU$ could be included
but lacks interest.)

\begin{thm}\label{thm3.3}
Assume that $(U,\cdot \; )$ is a totally ordered monoid with
absorbing element $0$ and $0<1,$ and that $p:U\to U$ is a map
obeying the rule {\rm $(\Gh1)$--$(\Gh5)$} above in Scholium~
\ref{schol3.2}.b. Assume also that the submonoid
$p(U)\setminus\{0\}$ of $(U,\cdot \; )$ is cancellative. Define an
addition $U\times U\overset +\rightarrow U$ by the rule given in
Scholium \ref{schol3.2}.c. Then $U$, enriched by this addition, is
a totally ordered supertropical semiring, and $p(x)=ex$ for all
$x\in U$, with $e:=e_U:=1_U+1_U.$ If $U\ne p(U)$ and $U\setminus
p(U)$ is closed under multiplication, then $U$ is a supertropical
predomain.
\end{thm}

\begin{proof} By \cite[Theorem 3.1]{IKR2} we have a unique structure of a
supertropical semiring on the set $U$ with multiplication as
given, such that $p=\nu _U,$ hence $p(U)=eU,$ the addition being
defined by the rule in Scholium \ref{schol3.2}.c. It is also clear
from the assumptions that $U$ is a supertropical predomain, if
$U\ne p(U)$ and $U\setminus p(U)$ is closed under multiplication.
Also the requirements \eqref{3.2} and \eqref{3.3} are covered by
the assumptions. Notice that up to now $(\Gh4)$ and $(\Gh5)$ are
not yet needed.

It remains to verify \eqref{3.1}. Thus, given $x,y,z\in U$ with
$x\le y,$ we have to prove that $x+z\le y+z.$ By $(\Gh4)$ we know
that $p(x)\le p(y).$ We distinguish the cases $p(x)<p(y)$ and
$p(x)=p(y)$ and go through various subcases.

\textit{Case 1}. $p(x)<p(y).$

\begin{enumerate}\item[a)] If $p(z)<p(x),$ then $x+z=x,$ $y+z=y.$

\item[b)] If $p(z)=p(x)$, then $x+z=p(x),$ $y+z=y.$ Suppose
$p(x)>y.$ Applying $p,$ we obtain $p(x)\ge p(y),$ a contradiction.
Thus $p(x)\le y.$ (In fact $p(x)<y.)$

\item[c)] If $p(x)<p(z)<p(y)$, then $x+z=z,$ $y+z=y,$ and $z<y$,
since $z\ge y$ would imply $p(z)\ge p(y).$

\item[d)] If $p(z)=p(y)$, then $x+z=z,$ $y+z=p(z),$ and $z\le
p(z)$ by $(\Gh5)$.

\item[e)] If $p(y)<p(z),$ then $x+z=z=y+z.$
\end{enumerate}
Thus in all subcases $x+z\le y+z.$

\textit{Case 2}. $p(x)=p(y).$

\begin{enumerate}\item[a)] If $p(z)<p(x),$ then $x+z=x,$ $y+z=y.$

\item[b)] If $p(z)=p(x)$, then $x+z=p(z)=y+z.$

\item[c)] If $p(z)>p(x) $, then $x+z= z=y+z.$
\end{enumerate}
Thus again in all subcases $x+z\le y+z.$
\end{proof}

We will need two more observations about totally ordered
supertropical semirings.

\begin{lem}\label{lem3.4}
Let $U$ be a totally ordered supertropical semiring.
\begin{enumerate}
\item[i)] If $x\in \cT(U),$ $y\in \cG(U),$ then
$$x\le y \Leftrightarrow ex\le y.$$
\item[ii)] If $x\in \cG (U),$ $y\in \cT(U),$ then
$$x\le y\Leftrightarrow x<ey.$$\end{enumerate}\end{lem}

\begin{proof}
i) $(\Rightarrow)$: $x\le y$ implies $ex\le ey=y.$

$(\Leftarrow)$: If $ex\le y$ then $x\le y$ (even $x<y),$ because
$x<ex.$ \pSkip

ii) $(\Rightarrow)$: Clear, because $y<ey.$

$(\Leftarrow)$: Let $x<ey.$ Suppose $x>y.$ Then $x=ex\ge ey,$ a
contradiction. Thus $x\le y$ (even $x<y).$
\end{proof}

We are ready to obtain an analogue of the construction of the
supertropical semirings STR $(\cT,\cG ,v)$ in \cite[Construction
3.16]{IKR1} for totally ordered supertropical semirings.

Assume that totally ordered monoids $(  \cT,\cdot \; )$, $(\cG
,\cdot \; )$ and a monoid homomorphism $v: \cT\to \cG $ are given.
Assume further that $(\cG ,\cdot \; )$ is cancellative and $v$ is
order preserving, i.e.,
$$\forall x,y\in \cT:x\le_\cT y\ds\Rightarrow v(x)\le_{\cG }v(y).$$ (Here $\le_\cT$ denotes the ordering of $\cT$ and
$\le_{\cG }$ denotes the ordering of $\cG .)$

By \cite[Construction 3.16]{IKR1} we have the supertropical
semiring
$$U:=
\STR(\cT,\cG ,v)$$ at our disposal. As a set, $U$ is the disjoint
union of $\cT,\cG $ and $\{0\}$, and $\cT$ is the set of tangible
elements of $U,$ whereas $\cG $ is the set of ghost elements of
$U.$ The multiplication on $U$ extends the multiplications on
$\cT$ and on $\cG ,$ has 0 as absorbing element, and obeys the
rule $x\cdot y=v(x)y$ for $x\in \cT,$ $y\in\cG .$ Moreover,
$e=v(1)$ and (hence) $ey=v(y)$ for all $y\in \cT.$\footnote{The
construction of $\STR(\cT,\cG ,v)$ can be understood as  a special
case of \cite[Theorem 3.1]{IKR2}.}

We want to define a total ordering on the set $U$ which extends
the given total orderings $\le_\cT$, $\le_{\cG }$ on $\cT$ and
$\cG $, and turns $U$ into a totally ordered supertropical
semiring. By \lemref{lem3.4} we are forced to define the relation
$\le$ on $U$ as follows. \{The label ``EO" alludes to ``extended
ordering".\}
\begin{align*}
& (\EO1): \quad   \forall x\in U:\quad  0\le x.  \\
& (\EO2): \quad  \text{If}\  x,y\in \cT:\quad x\le
y\Leftrightarrow
x\le_\cT y. \\
& (\EO3): \quad   \text{If}\ x,y\in \cG:\quad  x\le
y\Leftrightarrow x\le_{\cG }
y.  \\
& (\EO4): \quad  \text{If}\ x \in \cT,\ y\in\cG :\quad x\le
y\Leftrightarrow
v(x)\le_{\cG } y.  \\
& (\EO5): \quad   \text{If}\ x\in \cG ,\ y\in\cT :\quad x\le
y\Leftrightarrow x<_{\cG } v(y).  \\
\end{align*}

\begin{thm}\label{thm3.5}
The relation $\le$, defined by the rules {\rm$(\EO1)$--$(\EO5)$}
on $U:=\STR(\cT,\cG ,v),$ is a total ordering of this
supertropical semiring.
\end{thm}

\begin{proof}
We intend to apply \thmref{thm3.3}. For that reason, we regard $U$
as a multiplicative monoid, equipped with the map $$p:U\to U,
\qquad  x\mapsto ex.$$ Thus $p(x)=x$ if $x\in \cG \cup\{0\},$ and
$p(x)=v(x)$ if $x\in \cT.$ Clearly, $p$ has the properties
 $(\Gh1)$--$(\Gh3)$ listed in Scholium \ref{schol3.2}.b, and the addition
 of the semiring $U$ is given by the rule listed in Scholium
 \ref{schol3.2}.c, cf. \cite[Construction 3.16]{IKR1}.

 Our main task now is to verify that the binary relation $\le$ on
 $U$ defined by $(\EO1)$--$(\EO5)$ is a total ordering of the set $U$ and
 obeys the rules $(\Gh4)$ and $(\Gh5)$.

 We read off from the list $(\EO1)$--$(\EO5)$ that $x\le x$ for every
 $x\in U.$ Further
 \begin{align*}&\forall x,y\in U: \ x\le y,\ y\le x\Rightarrow
 x=y,\\
 &\forall x,y\in U: \ x\le y\  \text{or} \ y\le x.
 \end{align*}
 Finally, we obtain $(\Gh4)$ and $(\Gh5)$, i.e.,
 \begin{align*}
& \quad  x\le y\Rightarrow p(x)\le p(y),  \\
& \quad   x\le p(x).
\end{align*}
%
If $a,b\in U$ then we mean by $a<b$ that $a\le b$ and $a\ne b.$
\{N.B. It is not completely ridiculous to state this convention,
since we do not yet know that $\le$ is an ordering.\}

From $(\Gh4)$ we conclude that
\begin{equation}\label{3.4}
p(y)<p(x)\Rightarrow y<x.\end{equation} Indeed, if $y<x$ does not
hold then $x\le y,$ hence $p(x)\le p(y)$ contradicting
$p(y)<p(x).$

We are now ready to prove the transitivity of our relation. Let
$x,y,z\in U$ be given with $x\le y$ and $y\le z.$ We have to
verify that $x\le z.$ By $(\Gh4)$ above we have $p(x)\le p(y)$ and
$p(y)\le p(z),$ hence $p(x)\le_{\cG }p(y)$ and
$p(y)\le_{\cG}p(z).$ This implies $p(x)\le_{\cG } p(z),$ hence
$p(x)\le p(z).$ \pSkip

\textit{Case 1}. $p(x)<p(z).$ We conclude by \eqref{3.4} that $x<
z.$  \pSkip

\textit{Case 2}. $p(x)=p(z).$ We conclude from $p(x)\le_{\cG }
p(y)\le_{\cG } p(z),$ that also $p(x)=p(y). $ If ~$p(z)=0$ then
$x=0=z.$ Henceforth we assume that $p(z)\ne0.$ Now $x,y,z\in\cG
\cup \cT.$ We proceed through several subcases.

\begin{enumerate}
    \item[a)] If $x,y,z\in \cT$, then it is clear from the transitivity of the
relation $\le_{\cT } $ that $x\le z.$

\item[b)] If $x\in\cG $ then the relations $x\le y$ and
$p(x)=p(y)$ force $x=y$ by the rules $(\EO5)$ and $(\EO3)$. We
conclude from $y\le z$ that $x\le z.$

\item[c)] Similarly, if $y\in\cG $ we obtain $y=z$ and then
$x\le z.$

\item[d)] There remains the case that $x,y\in \cT$ and $z\in\cG .$ Since $p(x)=p(z)$, we learn from
$(\EO4)$ that $x\le z$ again.
\end{enumerate}

Thus $x\le z$ in all subcases.

We now know that our relation $\le$ is a total ordering on the set
$U$ obeying the rules $(\Gh4)$ and $(\Gh5)$. It extends the given
orderings on $\cG $ and $\cT.$

We check the compatibility of this ordering with multiplication.
Let $x,y,z\in U$ be given with $x\le y.$ We have to verify that
$x\cdot z \le y\cdot z.$

If $x=0$ or $z=0$ this is obvious. Thus we may assume that all
three elements $x,y,z$ are in $\cT\cup\cG .$ By $(\Gh4)$ we have
$p(x)\le p(y).$ Since our ordering restricted to $\cG $ is known
to be compatible  with multiplication, we conclude, using
$(\Gh1)$, that
$$p(xz)=p(x)p(z)\le p(y)p(z)=p(yz).$$
If $z\in\cG ,$ or both $x,y$ are in $\cG ,$ then $xz,yz\in\cG $
and thus
$$xz=p(xz)\le p(yz)=yz.$$
If all three elements are in $\cT,$ then $xz\le yz$, since the
restriction of our ordering to $\cT$ is known to be compatible
with multiplication.

We are left with the cases $x\in \cT,$ $y\in\cG $, $z\in \cT$ and
$x\in\cG ,$ $y\in \cT,$ $z\in \cT.$ In the first case, we have
$xz\in \cT, $ $yz\in\cG ,$ and we see by the rule $(\EO4)$ that
$xz\le yz.$ In the second case, $xz\in\cG ,$ $yz\in \cT,$ and we
see by $(\EO5)$ that $xz\le yz,$ using in an essential way that
the monoid $\cG $ is cancellative.

We now know that all assumptions made in \thmref{thm3.3} for the
totally ordered monoid $(U,\cdot \; )$ and the map $p:U\to U$ are
valid in the present case. Thus by this theorem our ordered monoid
$U$, together with the addition described in Scholium
\ref{schol3.2}.c, is a totally ordered supertropical semiring. But
this addition is the original one on the semiring $U=\STR(\cT,\cG
,v),$ cf. \cite[Constuction 3.16]{IKR1}. Our proof of
\thmref{thm3.5} is complete.
\end{proof}

\begin{defn}\label{defn3.6}
Given a triple $(\cT,\cG ,v)$ consisting of totally ordered
monoids $\cT,\cG $, with~ $\cG $ cancellative and an order
preserving monoid homomorphism from $\cT$ to $\cG ,$ we call the
supertropical predomain $U=\STR(\cT,\cG ,v)$ together with the
total ordering on $U$ by the rules $(\EO1)$--$(\EO5)$ the
\bfem{ordered supertropical semiring associated to} $(\cT,\cG
,v)$, and we denote this ordered supertropical predomain by
$\OSTR(\cT,\cG ,v).$\end{defn}

We want to interpret $m$-valuations as a special kind  of
supervaluations with values in such semirings $\OSTR(\cT,\cG ,v).$
For that reason we need some more terminology.

\begin{defn}\label{defn3.7}
Let $R$ be a semiring.
\begin{enumerate}
\item[a)] A \bfem{value-ordered supervaluation} on $R$, or
\bfem{vo-supervaluation} for short, is a supervaluation $\varphi:
R\to U$ with $U$ a totally ordered supertropical semiring.
\item[b)] We call a vo-supervaluation $\varphi: R\to U$
\bfem{ultrametric}, if $$ \forall a,b\in R:\
\varphi(a+b)\le\max(\varphi(a),\varphi(b)).$$ \item[c)] Let
$\varphi: R\to U$ and $\psi: R\to V$ be vo-supervaluations, and
let $U'$, $V'$ denote the subsemirings of $U$ and $V$ generated by
$\varphi(R)$ and $\psi(R)$ respectively, i.e., $U'=\varphi(R)\cup
e\varphi(R),$ $V'=\psi(R)\cup e\psi(R)$ (cf. \cite[Proposition
4.2]{IKR1}). We call $\varphi$ and $\psi$ \bfem{order-equivalent},
and write $\varphi \simeq_0\psi,$ if there exists an
order-preserving isomorphism $\al : U'\osr V'$ with $\psi(a)=\al
(\vrp(a))$ for every $a\in R.$
\end{enumerate}
\end{defn}

\begin{remarks}\label{remarks3.8}$ $
\begin{enumerate}
\item[a)] Given a vo-supervaluation $\varphi: R\to U$, let $v:
R\to eU$ denote the $m$-valuation\footnote{In the examples below
$v$ will be a valuation.} covered by the supervaluation $\varphi$.
Then, for every $a\in R,$
$$\varphi(a)\le e\varphi(a)=v(a).$$
The \bfem{support} of $\varphi$, defined as
$$\supp(\varphi):=\{a\in R\ds|\varphi(a)=0\},$$
coincides with the support of $v.$ \item[b)] If $\varphi$ is
ultrametric then
$$A_\varphi:=\{a\in R \ds |\varphi(a)\le 1\}$$
is a CMC-subsemiring of $R$ contained in $A_v,$ and
$$\mfp_\varphi:=\{a\in R\ds |\varphi(a)<1\}$$
is a prime of $R$ contained in $\mfp_v.$
\end{enumerate}
\end{remarks}

\begin{construction}\label{constr3.9}
Let $w:R\to N$ be an $m$-valuation and $\rho: N\to M$ a semiring
homomorphism from $N$ to a \bfem{cancellative} bipotent semiring
$M.$ Assume that $\rho^{-1}(0)=\{0\}.$ Thus $v:= \rho \circ w$ is
a valuation coarsening $w,$ and $v,$ $w$ have the same support
$v^{-1}(0)=w^{-1}(0).$ Moreover, $M=\cG  \cup\{0\}$ and $N
=\cT\cup\{0\}$ with $\cG $ a totally ordered cancellative
(multiplicative) monoid and also $\cT:=\rho^{-1}(\cG )$ a totally
ordered monoid.

Abusing notation, we denote the monoid homomorphism $\cT\to \cG $
obtained from $\rho$ by restriction again by $\rho.$ It is order
preserving. Then in  the totally ordered supertropical predomain
$$U:=\OSTR(\cT,\cG ,\rho)$$
we identify $\cT=\cT(U),$ $\cG =\cG (U)$ and then
$N=\cT\cup\{0\}$, $M=\cG \cup\{0\}$ in the obvious way. Now $M=eU$
and $N = U\setminus\cG $. The map
$$\varphi: R\to U,\quad a\mapsto w(a)\in N\subset U,$$ sends $0$ to
$0$, $1$ to $1$, and is multiplicative. Further $e\varphi(a)=v(a)$
for all $a\in R.$ Thus $\varphi$ is a supervaluation covering the
valuation $v.$ It has values in $N=\cT(U)\cup\{0\}$, and
$$\varphi(a+b)\le\max(\varphi(a),\varphi(b))$$
for $a,b\in R$, since the ordering of $U$ extends the ordering on
$N$ and $w$ is an $m$-valuation. We conclude that $\varphi: R\to
U$ is a tangible ultrametric supervaluation.
\end{construction}

Conversely, given a tangible ultrametric supervaluation $$\varphi:
R\to \OSTR(\cT,\cG ,\rho)=U,$$ we may view $\varphi$ as a map $w$
from $R$ to $N=\cT(U)\cup\{0\}$, and this is an $m$-valuation.
Moreover, $e\varphi=v.$ Extending $\rho: \cT\to \cG $ to a
semiring homomorphism $\rho: N\to M$ by $\rho(0)=0$, we have
$\rho\circ w=v.$

Now the following is fairly obvious.

\begin{thm}\label{thm3.10}
Given a valuation $v: R\to M,$ the $m$-valuations $w:R\to N$
dominating $v$ (cf. \cite[\S2]{IKR1}) and having the same support
as $v$ correspond with the tangible ultrametric supervaluations
$\varphi: R\to U$ covering $v$ (i.e., $eU=M,$ $ e\varphi=v$)
uniquely up to order equivalence in the way indicated by
Construction \ref{constr3.9}.
\end{thm}

\begin{examples}\label{examps3.11} $ $

\begin{enumerate} \eroman
    \item  Let $R$ be a semiring and $\mfp$ a prime of $R.$ In \S\ref{sec:1} we
defined the associated $V^0$-valuation
$$w:=v_{\mfp}:R\to N:=M(R,\mfp),$$
and in \S\ref{sec:2} we established the semiring homomorphism
$$\rho: = \pi_C:N\to M:=N/C.$$
We learned that $\rho^{-1}(0)=\{0\}.$ Applying Construction
\ref{constr3.9} to these  data, we obtain a tangible ultrametric
supervaluation
$$\varphi:=\varphi_{\mfp}: R\to U(R,\mfp)$$
which is determined by the pair $(R,\mfp)$ above and covers the
valuation
$$v:=w/C: R\to M.$$
Here $U(R,\mfp)$ denotes the totally ordered supertropical
semiring $\OSTR(\cT,\cG ,\rho)$ from Construction \ref{constr3.9}.
We  have
$$\mfp_\varphi=\mfp,\qquad A_\varphi=[\mfp:\mfp].$$ \{Recall the notations in Remark
\ref{remarks3.8}.\} \pSkip

\item  Similarly, given a proper CMC-subsemiring $A$ of $R,$ we
obtain from the associated $V$-valuation
$$w:=v_A: R\to N:=M(R,A)$$
a tangible ultrametric supervaluation
$$\varphi: =\varphi_A:R\to U(R,A)$$
covering $w/C,$ with $A_\varphi=A$ and
$$\mfp_\varphi = P(A):= \{x\in R \ds|\exists y\in R\setminus A:
xy\in A\}.$$

\end{enumerate}

\end{examples}

\begin{remark}\label{rem3.12}
Assume that $\varphi: R\to U$ is a tangible ultrametric
supervaluation covering a valuation $v: R\to M.$ Now choose an
MFCE-relation $E$ on $U$ which is also compatible with the
ordering on $U$ (cf. \cite[\S4]{IKR2}). Then $U/E$ is again a
totally ordered supertropical  semiring and
$$\varphi/E :=\pi_E\circ\varphi: R\to U/E$$
is again an ultrametric supervaluation covering $v$ (cf.
\S\ref{sec:5} below for more details). But often~ $\varphi/E$ will
not be tangible. Then $\psi: =\varphi/E$ cannot be interpreted as
just an $m$-valuation covering $v.$

 Of course, $\psi(R)$ is a
multiplicative monoid with absorbing element 0, and $\psi(R)$ is
totally ordered by the ordering of $U.$ Thus the map $\psi: R\to
\psi(R)$ may be viewed as an $m$-valuation, but doing so we loose
information about the supervaluation $\psi.$\end{remark}

We hasten to exhibit the ``simplest" tangible ultrametric
supervaluations.

\begin{example}\label{examps3.13} Let $v: R\to \cG \cup\{0\}=M$ be a surjective valuation. Take $w=v$ in
Construction \ref{constr3.9}. We obtain a tangible ultrametric
supervaluation
$$\varphi: R\to U:=\OSTR(\cG,\cG,\id_{\cG}).$$ The supertropical domain
$$D(\cG):=\STR(\cG,\cG,\id_{\cG})$$
has been described in \cite[\S3]{IKR1}. The minimal ordering of
$D(\cG )$ is a total ordering. Thus we can identify $D(\cG )=U.$

The vo-supervaluation $\varphi$ coincides with the supervaluation
$\chv: R\to D(\cG )$ in \cite[Example 9.16]{IKR1}. It is the
minimal tangible supervaluation covering $v.$\end{example}

More generally, it is fairly obvious that the minimal ordering of
a supertropical predomain~ $U$ is total iff every fiber of the
ghost map $\nu_U$ contains at most one tangible element.

\section{Supervaluations from generalized CMC-sets}\label{sec:4}

In this section $R$ is a semiring.

\begin{defn}\label{defn4.1} $ $
\begin{enumerate}\item[a)]  A \bfem{CMC-subset of} $R$ is a set
$A\subset R$ such that $0\in A,$ $1\in A,$ both $A$ and
$R\setminus A$ are closed under multiplication, and there exists a
unit $u$ of $R$ with
$$u(A+A)\subset A.$$
\item[b)] We call any such unit $u$ an \bfem{exponent of} $A.$
\item[c)] The CMC-subsemirings of $R$ are the CMC-subsets that
have exponent 1. If $A$ does not admit exponent 1, we call $A$ a
\bfem{true CMC-subset} of $R.$ This means that $A$ is not a
subsemiring of $R.$ In particular, then $A\ne
R.$\end{enumerate}\end{defn}

Essentially this is the terminology of Valente and Vitulli in
their paper \cite{VV}, which in turn is rooted in the terminology
of Harrison and Vitulli in \cite{HV1}, \cite{HV2}. But we slightly
deviate from \cite{VV}. Valente and Vitulli call our CMC-subsets
``weak CMC-subsets" and our exponents ``weak exponents". They
define CMC-subsets (without ``weak") by including
 still one additional property of an archimedean flavor,
following the route developed by Harrison and Vitulli in their
quest for ``infinite primes", which are generalizations of the
classical archimedean primes in number fields.

For our purposes here, to find interesting new examples of
supervaluations, it will be amply clear that CMC-subsets as
defined above should be the basic structure. \{Consequently, in a
planned extension of this paper  we will call the CMC-subsets and
exponents of \cite{VV} ``strong CMC-subsets" and ``strong
exponents".\}

Valente and Vitulli deal only with CMC-subsets in rings. They
speak of ``nonring CMC-subsets" instead of our ``true
CMC-subsets". The analogous terms ``nonsemiring CMC-subsets" would
be simply too long.

In the papers \cite{HV2}, \cite{VV}, an exponent is most often
denoted by the letter ``e". We have to deviate also from this
habit due to our permanent use of ``e" for the ghost unit element
of a supertropical semiring.

\begin{examples}\label{examps4.2}
Let $R$ be a totally ordered field.
\begin{enumerate}
    \item[i)] The closed unit interval
$$[-1,1]_R:=\{x\in R \ds|-1\le x\le 1\}$$
is a true CMC-subset of $R$ with exponent $\frac{1}{2}.$

\item[ii)] The closed unit disk
$$\{x+iy \ds |x^2+y^2\le 1\}$$
of the field $R(i),$ $i:=\sqrt{-1},$ is a true CMC-subset of
$R(i),$ again with exponent $\frac{1}{2}.$
\end{enumerate}
\end{examples}

\begin{example}\label{examps4.3}
Let $U$ be a totally ordered supertropical semiring which is not
ghost, i.e., $U\ne eU.$ Then
$$A:=A_U:=\{x\in U \ds|x\le 1\}$$
is closed under multiplication and $0,1\in A.$ Also $R\setminus A$
is closed under multiplication. Assume that $U$ has a unit $u$
(necessarily tangible) with $eu<1,$ a rather mild condition. Then
$u(A+A)\subset A .$ Indeed, for $x,y\in A$ we have $eux<1,$
$euy<1,$ hence
$$u(x+y) \leq  eu(x+y)=\max(eux;euy)<1.$$
But $1+1 \notin  A.$ Thus $A$ is a true CMC-subset of $U.$

If such a unit $u$ does not exist, then $A:=A_U$ is not a
CMC-subset of $U.$ Indeed, for any unit $u'$ of $U$ with
$u'(A+A)\subset A,$ we infer that $eu'=u'(1+1)\in A,$ i.e.,
$eu'\le 1,$ which forces~ $eu'<1.$\end{example}

In the following $R$ is a semiring and $A$ is a true CMC-subset of
$R.$ In \S\ref{sec:1} we constructed $m$-valuations
$$v_B: R\twoheadrightarrow M(R,B),\qquad v_{\mfp}: R \twoheadrightarrow M(R,\mfp)$$
for $B$ a proper CMC-subsemiring of $R$ and -- with more detailed
arguments -- for $\mfp$ a prime of~$R.$

We now will find, proceeding exactly in the same way, a map $v_A:
R\twoheadrightarrow M(R,A)$ with $M(R,A)$ again a bipotent
semidomain, but $v_A$ will be a multiplicative map showing a
behavior under addition somewhat weaker than $m$-valuations do.

\begin{prop}\label {prop4.4}
If $x,y\in R$ and $[A:x]\not\subset[A:y],$ then
$[A:y]\subset[A:x].$
\end{prop}

\begin{proof} Argue as in the proof of \propref{prop1.5} by using
\lemref{lem1.4} with $L=A.$\end{proof}

The proposition gives us an equivalence relation $\sim_{R,A}$ on
the set $R,$ defined by
$$x\sim_{R,A}y \ \Leftrightarrow \ [A:x]=[A:y],$$
and then a total ordering on the set
$$M(R,A):=R/\sim_{R,A}$$
of equivalence classes, given by
$$[x]\le[y] \ \Leftrightarrow \ [A:x]\supset[A:y]$$
where we denote the equivalence class of an element $z\in R$ by
$[z],$ or more precisely $[z]_{R,A}$ if necessary. The equivalence
relation $\sim_{R,A}$ turns out to be multiplicative; hence we
have  a well-defined multiplication  on $M(R,A)$ given by
$[x]\cdot[y]:=[xy].$ It has the unit element $[1],$ and the
ordering is compatible with multiplication. Moreover, $[0]$ is the
smallest element of $M(R,A)$. The set $M(R,A)\setminus\{[0]\}$
turns out to be closed under multiplication. Indeed, for $x\in R$
we have $[x]=[0]$ iff $[A:x]=[A:0]=R,$ i.e., $Rx\subset A.$ Thus,
if $[x]\ne0,$ $[y]\ne0,$ we have elements $s,t\in R$ with
$sx\notin A,$ $ty\notin A,$ hence $(st)(xy)\notin A,$ hence
$[xy]\ne [0].$

Thus we may regard $M(R,A)$ as a bipotent semidomain.

\begin{thm}\label{thm4.5}
{}\quad

\begin{enumerate}
\item[i)] The map
$$v:=v_A:R\to M(R,A),\qquad x\mapsto [x],$$
is multiplicative, and $v(0)=0,$ $v(1)=1.$

\item[ii)]
$v^{-1}(0)=\{ x\in R|Rx\subset A\}.$

\item[iii)] If $u$ is an
exponent of $A$ then, for all $x,y\in R$
\begin{equation}\label{4.1}
v(x+y)\le v(u^{-1})\cdot \max(v(x),v(y)).\end{equation} \item[iv)]
$\{x\in R|v(x)\le 1\}=A.$ \item[v)] $\{x\in R|v(x)<1\}=P(A)$ with
\begin{equation}\label{4.2}
P(A):=\{x\in R \ds|  \exists y\in R\setminus A: xy\in
A\}.\end{equation}
\end{enumerate}\end{thm}

\begin{proof}
i) and ii) are obvious from the above. \pSkip

iii): We may assume that $v(x)\le v(y),$ i.e., $[A:x]\supset
[A:y].$ Let $z\in [A:y],$ then also $z\in[A:x].$ Since both
$zx,zy$ are in $A$, we conclude that
$$uz(x+y)\in u(A+A)\subset A.$$
Thus $z\in [A:u(x+y)].$ This proves that $[A:y]\subset
[A:u(x+y)],$ in other terms, $$v(u(x+y)\le v(y).$$ Since $v$ is
multiplicative, we conclude that $v(x+y)\le v(u^{-1})v(y).$ \pSkip

iv): $v(x)\le 1$ iff $[A:x]\supset [A:1]=A$ iff $Ax\subset A$ iff
$x\in A.$ \pSkip

v): $v(x)<1$ iff $x\in A$ but $[A:x]\ne A$ iff there exists some
$s\in R\setminus A$ with $sx\in A.$\end{proof}

In \S\ref{sec:1} we have seen that in much the same way primes and
CMC-subsemirings of $R$ give us $m$-valuations, namely $\Vz$- and
$V$-valuations. In the present context we have a similar story
dealing with CMC-subsets of $R$ and ``prime subsets" of $R,$ to be
defined now.

\begin{defn}\label{defn4.6} A \bfem{prime subset of} $R$ is a set
$\mfp\subset R$ such that $0\in\mfp,$ $1\notin \mfp$, both $\mfp$
and~$R\setminus\mfp$ are closed under multiplication, and there
exists a unit $u$ of $R$ with
$$u(\mfp+\mfp)\subset\mfp.$$
We call any such unit $u$ an \bfem{exponent} of $\mfp.$

The primes of $R$ are the prime subsets of $R$ which have exponent
1. The other prime subsets will be called the \bfem{true prime
subsets} of $R.$
\end{defn}

\begin{examp}\label{examp4.7}
Let $A$ be a CMC-subset of $R $ with exponent $u.$ Then
$$P(A):=\{x\in R \ds |  \exists y\in R\setminus A:xy\in A\}$$
turns out to be a prime subset of $R$ with exponent $u.$ This
follows easily from the description
$$P(A)=\{x\in R \ds|  v_{R,A}(x)<1\}$$
in \thmref{thm4.5} and the properties of $v_{R,A}$ stated in parts
i) and iii) of that theorem.\end{examp}

We call $P(A)$ the \bfem{central prime set} of the CMC-set $A$ (in
$R).$

\begin{examples}\label{examps4.8}
Let $R$ be a totally ordered field.
\begin{itemize}
    \item[a)] The open unit interval
$$]-1,+1[_R:=\{x\in R \ds|-1<x<1\}$$
is a true prime subset of $R$ with exponent $\frac{1}{2}.$ It is
the central prime set of $[-1,1]_R$.

\item[b)] The open unit disk
$$\{x+iy \ds |x^2+y^2<1\}$$
of the field $R(i)$, $i=\sqrt{-1},$ is a true prime subset of
$R(i)$ with exponent $\frac{1}{2}.$ It is the central prime set of
the closed unit disk of $R(i).$
\end{itemize}
\end{examples}

\begin{example}\label{examp4.9}
Let $U$ be a totally ordered supertropical semiring. Assume that
there exists a unit $u$ of $U$ with $eu<1.$ Then
$$\mfp_U:=\{x\in U \ds |x<1\}$$ is a prime subset of $U$ with
exponent $u.$ It is the central prime subset of the CMC-subset
$A_U$ of $U$ discussed in Examples \ref{examps4.3}.
\end{example}

Let $R$ be a semiring, as before, and $\mfp$ a prime subset of $R$
with exponent $u.$ Then we see exactly as above that for any
$x,y\in R$ either $[\mfp:x]\subset[\mfp:y]$ or
$[\mfp:y]\supset[\mfp:x],$ and we obtain a map
$$v:=v_{\mfp}:R\twoheadrightarrow M(R,\mfp)$$
onto a bipotent semidomain $M(R,\mfp)$ such that
$$v(x)\le v(y) \ds \Leftrightarrow [\mfp:x]\supset[\mfp:y].$$ $M(R,\mfp)$ is obtained from $R$ by dividing out the
equivalence relation given by
$$x\sim y \ds \Leftrightarrow [\mfp:x]=[\mfp:y],$$ and
$v(x)$ is the equivalence class of $x$ in this relation.

Parallel to \thmref{thm4.5}, we have the following facts to be
proved in an analogous way.

\begin{thm}\label{thm4.10}
\quad{}

\begin{enumerate}\item[i)] The map $v:=v_{\mfp}$ is
multiplicative and $v(0)=0,$ $v(1)=1.$

\item[ii)] $v^{-1}(0)=\{x\in R| \ Rx\subset \mfp\}.$

\item[iii)] If $u$ is an exponent of $\mfp$ then, for all
$x,y\in R$
\begin{equation}\label{4.3}
v(x+y)\le v(u^{-1})\cdot\max(v(x),v(y)).\end{equation}

\item[iv)]$ \{x\in R| \ v(x)\le 1\}=A(\mfp)$ with
\begin{equation}\label{4.4}
A(\mfp):=\{x\in R| \ yx\subset\mfp\}=[\mfp:\mfp]\end{equation}

\item[v)] $ \mfp = \{x\in R| \ v(x)<1\}=\{ x\in R \ds |  \exists y \in R \setminus \mfp : xy \in \mfp  \}$. \end{enumerate}\end{thm}

\begin{cor}\label{cor4.11}
$A(\mfp)$ is a CMC-subset of $R$ with exponent $u$.
\end{cor}

\begin{proof} This follows from points i), iii), iv) in the theorem.
\end{proof}

\begin{defn} We call a map $v: R \to R' $ from $R$  to a semiring
$R'$ \textbf{0-1-multiplicative} if $v$ is multiplicative, i.e.,
$$ \forall x,y \in R: \quad v(xy) = v(x) \cdot v(y),$$
and $v(0)=0$, $v(1)=1$.
\end{defn}

In the following 0-1-multiplicative maps from $R$ to bipotent
semirings which are not m-valuations, will play  a major role. We
already met such maps in Theorem \ref{thm4.5} and \ref{thm4.10}.
The following remark is sometimes useful

\begin{remark}\label{rmk4.13}
Every surjective multiplicative map   $v: R \to R'$ form $R$ to a
semiring  $R'$ is 0-1-multiplicative.
\end{remark}

\begin{proof} Let $z \in R' $ be given. We choose some $x\in R$
with $v(x) = z$. Then
$$
\begin{array}{lllll}
v(1) \cdot z  & =  v(1) \cdot v(x) & =  v(1 \cdot x) &= v(x) &  =
z, \\[2mm]
 v(0) \cdot z & =  v(0) \cdot v(x) & =  v(0 \cdot x) &  = v( 0).& \end{array}
$$
From $ v(1) \cdot z = z$ for every $z \in R'$ we  conclude that
that $v(1)=1.$ From $ v(0) \cdot z = 0$ for every $z \in R'$ we
conclude that  $v(0)=0$, since otherwise we  would have $z \neq 0$
for every $z \in R'$, which is not true.
\end{proof}

We extend the notion of dominance for $m$-valuations on $R$
\cite[\S 2]{IKR1} to 0-1-multiplicative maps from $R$ to bipotent
semirings.

\begin{defn}\label{defn4.13}  If $v:R\to
M$ and $w:R\to N$ are 0-1-multiplicative maps  from $R$ to
bipotent semirings $M,N$, we say that $v$ \bfem{dominates} $w,$
or, that $w$ is a \bfem{coarsening of} $v,$ if
$$\forall a,b\in R: \qquad  v(a)\le v(b) \ \Rightarrow \ w(a)\le w(b).$$
We then write $v \geq w.$
\end{defn}

The following proposition is now obvious.
\begin{prop}\label{prop4.15}
Assume that  $v:R\to M$ and $w:R\to N$ are 0-1-multiplicative maps
from $R$ to bipotent semirings $M,N$. Assume further that $v$ is
surjective. Then $v \geq w$  iff there exists a (unique) semiring
homomorphism  $\gm: M \to N$ with $\gm \circ v = w$.

\end{prop}

As in \cite[\S 2]{IKR1} we denote this homomorphism $\gm$ by
$\gm_{w,v}$ if necessary. \pSkip

CMC-subsets and prime subsets of $R$ will lead  to many examples
of dominance of 0-1-multiplicative maps to bipotent semirings.
Since we will work with these subsets at an equal pace, we coin
the following notion.

\begin{defn}\label{defn4.14} A \bfem{generalized CMC-subset} $L$
of $R$ is either a CMC-subset $A$ or a prime subset $\mfp$ of $R.$
We call $L$ a \bfem{true} CMC-subset of $R$ if $L$ is neither a
subsemiring nor a prime of $R.$\end{defn}

Given a true generalized CMC-subset $L$ of $R$, we want to
interpret the surjective multiplicative  map $v_L$ as a tangible
value-ordered supervaluation in much the same way as we did this
in \S\ref{sec:3} for $L$ a CMC-subsemiring or a prime $\mfp$ of
$R.$

Two strategies come to mind:
      Find either a
proper CMC-subsemiring $B\supset L$ or   a prime $Q\subset L$ of
$R$ such that $v_L$ dominates the $m$-valuation $w:=v_B$  or $w:=
v_Q$ respectively!
  Then $v_L$ will also dominate the
\bfem{valuation} $w/C$ introduced in \S\ref{sec:2}.
 If, in addition, $v^{-1}(0)=w^{-1}(0),$ we can obtain the desired vo-supervaluation
by a straightforward generalization of
Construction~\ref{constr3.9}.

It will turn out that both strategies usually work well. Starting
from our true generalized CMC-subset $L$ of $R$ with exponent $u$,
we define the sets
$$B:=B_u(L):=\bigcup_{n\in\mathbb N}u^{-n}L,$$
$$Q:=Q_u(L):=\bigcap_{n\in\mathbb N}u^{n}L.$$

\begin{thm}\label{thm4.15}
{}\quad

\begin{enumerate}\item[i)] $B$ is a CMC-subsemiring of $R$
containing $L,$ and $Q$ is a prime subset of $R$ contained in $L.$

\item[ii)] $Q$ is a prime ideal of $B.$ \item[iii)] The
multiplicative map $v_L:R\twoheadrightarrow M(R,L)$ dominates the
$V^0$-valuation $v_Q: R\twoheadrightarrow M(R,Q),$ and
$v_L^{-1}(0)=v_Q^{-1}(0).$

\item[iv)] If $B\ne R,$ then $v_L$ dominates the $V$-valuation
$v_B: R\twoheadrightarrow M(R,B)$, and
$v_L^{-1}(0)=v_B^{-1}(0).$\end{enumerate}\end{thm}

\begin{proof}
a) It follows from $L\cdot L\subset L$ that $B\cdot B\subset B.$
Given $n\in\mathbb N$ we have
$$u^{-n}L+u^{-n}L=u^{-n-1}u(L+L)\subset u^{-n-1}L.$$
Thus $B+B\subset B.$ Further $1=u^{-1}\cdot u\in u^{-1}L, $ hence
$1\in B.$ Of course, $L\subset B$ and in particular $0\in B.$

Let $x,y\in R\setminus B.$ For every $n\in\mathbb N$ we have
$u^nx\notin L,$ $u^ny\notin L,$ hence $u^{2n}(xy)\notin L,$ hence
$xy\in R\setminus B.$ Altogether this proves that $B$ is a
CMC-subsemiring of $R$ containing $L.$ \pSkip

b) It follows from $L\cdot L\subset L$ that $L\cdot Q\subset Q.$
Of course, $Q\subset L.$ Thus certainly $Q\cdot Q\subset Q.$ Let
$x,y\in Q$ be given. For every $n\in\mathbb N, $ we have $x\in
u^{n+1} L,$ $y\in u^{n+1}L,$ hence
$$x+y\in u^{n+1}(L+L)\subset u^nL.$$
Thus $x+y\in Q.$ Of course, $0\in Q.$ On the other hand, $1\notin
uL, $ hence $1\notin Q.$

Let $x,y\in R\setminus Q.$ There exists some $n\in \mathbb N$ with
$u^{-n}x\notin L,$ $u^{-n}y\notin L.$ Then $u^{-2n}(xy)\notin L.$
Thus $xy\in R\setminus Q.$ Altogether we have proved that $Q$ is a
prime of $R,$ contained in the set $L.$ \pSkip

c) Let $x\in Q$ be given. For any $n\in\mathbb N $ we have $x\in
u^{n+1}L,$ hence $u^{-1}x\in u^nL.$ Thus $u^{-1}x\in Q.$ This
proves that $u^{-1}Q\subset Q.$ Since also $L\cdot Q\subset Q$, we
see that $B\cdot Q\subset Q.$ Because $Q$ is a prime of $R$, it is
now obvious that $Q$ is a prime ideal of $B.$ \pSkip

 d) For any
$x\in R$ we have
$$[B:x]=\bigcup_n[u^{-n}L:x]=\bigcup_nu^{-n}[L:x].$$
Therefore
$$\forall x,y\in R: \quad [L:x]\supset[L:y]\Rightarrow
[B:x]\supset[B:y].$$ If $R\ne B$ this translates to
$$\forall x,y\in R: \quad v_L(x)\le v_L(y)\Rightarrow v_B(x)\le
v_B(y).$$ Thus $v_L$ dominates $v_B.$ \pSkip

 e) For any $x\in R$
we have
$$[Q:x]=\bigcap_n[u^nL:x]=\bigcap_nu^n[L:x].$$
It follows that
$$\forall x,y\in R: \quad [L:x]\supset[L:y]\Rightarrow
[Q:x]\supset[Q:y].$$   Thus $v_L$ dominates $v_Q.$ \pSkip

 f) As
observed before (Theorems \ref{thm4.5}, \ref{thm4.10}),
$$v_L^{-1}(0)=\{x\in R\ds|Rx\subset L\}.$$
Similarly $v_Q^{-1}(0)=\{x\in R \ds |Rx\subset Q\}$ and if $B\ne
R,$ $v_B^{-1}(0)=\{x\in R \ds |Rx\subset B\}.$ Thus
$v_Q^{-1}(0)\subset v_L^{-1}(0)$ and, if $B\ne R,$
$v_L^{-1}(0)\subset v_B^{-1}(0).$

We want to prove equality of these sets. Let $x\in R$ be given
with $Rx\not\subset  Q.$ Choose some $z\in R$ with $zx\notin Q.$
Then $u^{-n}zx\notin L$ for some $n.$ Thus $Rx\not\subset L.$ This
proves that $v_Q^{-1}(0)=v_L^{-1}(0).$

Now assume that $B\ne R.$ Let $x\in R$ be given with $Rx \not
\subset  L$. We choose elements $z,s$ of $R$ with  $xz \notin L$,
$s \notin B.$ Then $u^ns\notin L$ for all $n,$ hence $u^nszx\notin
L$ for all $n.$ It follow that $Rx\not\subset B.$ This proves that
$v_L^{-1}(0)=v_B^{-1}(0).$
\end{proof}

We describe all quantities occurring in Theorems \ref{thm4.5},
\ref{thm4.10} and \ref{thm4.15} in the perhaps simplest case of
interest.

\begin{examples}\label{examps4.16}
 Let $R$ be a real closed field. Associated to the
ordering of $R$, we have the absolute value $|\ |_R$ on $R$ with
$|x|_R=x$ if $x\ge0$ and $x _R=-x$ if $x\le 0.$ Also $R$ contains
(a unique copy of) the field $\mathbb R$ of real numbers. We
consider the CMC-subset
$$A:=\{x\in R \ds | \,  |x|_R\le 1\}=[-1,1]$$
and the prime subset
$$\mfp:=\{x\in R \ds |\, |x|_R<1\}$$
(as in Examples \ref{examps4.2}.i) of $R$, both with exponent
$u=\frac{1}{2}.$ Then
$$B_u(A)=B_u(\mfp)=\bigcup_n[-2^n,2^n] = \{x\in R \ds |\exists
n\in\mathbb N:|x|\le n\}$$ and $$Q_u(A)=Q_u(\mfp)=\bigcap
_n[-2^{-n},2^{-n}]=\Big\{x\in R \ds |\forall n\in \mathbb
N:|x|\le\frac{1}{n}\Big\}.$$ Thus, $B:=B_u(A)$ is a valuation
domain with quotient field $R$, and is also the smallest convex
subring of the ordered field $R,$ and $Q:=Q_u(A)$ is the maximal
ideal of $B.$ We have $B\ne R$ iff $\mathbb R\ne R.$

Observe that $$[A:x]\supset [A:y] \ \text{ iff } \ |x|_R\le |y|_R
\ \text{ iff } \ [\mfp:x]\supset[\mfp:y].$$ Thus $v_A=v_{\mfp},$
and we can identify this multiplicative map with the absolute
value map $$ x \mapsto |x|_R, \qquad R\twoheadrightarrow
R_{\ge0}.$$

Let $w$ denote the canonical valuation associated to $B.$ This is
the natural map
$$w: R\twoheadrightarrow R/B^*=R^*/B^*\cup\{0\},$$
with $R^*/B^*$ ordered by the rule
$$xB^*\le  y B^*\Leftrightarrow \frac{x}{y}\in B.$$
We have $B=\{x\in R \ds |w(x)\le 1\}$ and $Q=\{x\in R \ds
|w(x)<1\},$ and obtain for $x,y\in R$ in the case $B\ne R$ that
$$\text{$[B:x]\supset [B:y]$ \ iff \  $w(x)\le w(y)$ \ iff \
$[Q:x]\supset[Q:y].$}$$ Thus $v_B=v_Q\sim w.$ If $B=R$ we may
still identify $v_Q$ with the now trivial valuation $w.$

It is easily checked that $A(\mfp)=A$ and
$P(A)=\mfp.$\end{examples}

 The coincidences $B_u(A)=B_u(\mfp),$ $Q_u(A)=Q_A(\mfp),$ $A(\mfp)=A,$ $P(A)=\mfp$ in this example are
 typical for the case that the semiring $R$ is a semifield, or at
 least has ``many units" in an appropriate sense. We will pursue
 the case of semifields in \cite{IKR3}, while in the present paper
 we deal with the situation where such coincidences
 often fail.

\begin{construction}
 Given a surjective multiplicative map $w:R\twoheadrightarrow N$
 from $R$ to a bipotent semiring $N$ and a surjective valuation $v:
 R\twoheadrightarrow M$ such that $w$ dominates $v$ we have a
 semiring homomorphism $\rho: N\to M$ with $v=\rho\circ w.$ If
 also $v^{-1}(0)=w^{-1}(0)$, we can repeat Construction
 \ref{constr3.9} word by word. We obtain again an ordered
 supertropical predomain
 $$U:=\OSTR(\cT,\cG ,\rho)$$
 with $\cT=N\setminus\{0\}$, $\cG =M\setminus \{0\}$, and
 $\rho: \cT\to \cG $ gained from $\rho:N\to M$ by
 restriction. As in~\S\ref{sec:3} we identify $N=\cT\cup\{0\}\subset U$ and
 $M=\cG \cup\{0\}=eU.$ The map
 $$\varphi: R\to U,\qquad a\mapsto w(a)\in N\subset U,$$
 is a tangible supervaluation covering $v,$ but now -- in contrast
 to the situation in Construction~\ref{constr3.9} -- $v$ has no
 reason to be ultrametric. \endbox \end{construction}

\begin{examples}\label{examps4.17} $ $
\begin{enumerate}
    \item[a)] As a consequence of \thmref{thm4.15} we can apply this
construction for $L$  a true generalized CMC-subset of $R$ to
$w:=v_L:R\to M(R,L)$ and to $$v:=v_{Q}/C:R\to M(R,Q)/C.$$

 In this
case we denote the totally ordered supertropical semiring from
above by $U(R,L,Q_u(L))$ and the arising tangible
vo-supervaluation by
$$\varphi_{L,u}:R \twoheadrightarrow U(R,L,Q_u(L)).$$

\item[b)] Likewise, if $B_u(L)\ne R,$ we can take $w:=v_L$ and
$v:=v_B/C$ and obtain again a tangible  vo-supervaluation, which
we denote by
$$\psi_{L,u}:R\twoheadrightarrow U(R,L,B_u(L)).$$ \end{enumerate}

\end{examples}

If $\varphi: R\to U$ is any of these vo-supervaluations
$\varphi_{L,u},\psi_{L,u}$, then $\varphi$ obeys a rule
\begin{equation}\label{4.6} \forall a,b \in R: \quad
\varphi(a+b)\le c \max(\varphi(a),\varphi(b))\end{equation} with
$c$ a unit of $U,$ namely, $c =\varphi(u^{-1})$ for the chosen
exponent $u$ of $L.$ This follows from Theorems \ref{thm4.5}.iii
and \ref{thm4.10}.iii. \{N.B. Any unit of $U$ is a tangible
element.\}

\begin{defn}\label{defn4.18} The rule \eqref{4.6} reminds us of
the classical absolute values of Emil Artin (which he called
``valuations", cf. \cite[Chapter 1]{A1},  \cite[Chapter 3]{A2}).
Thus, we call a vo-valuation $\varphi$ obeying  \eqref{4.6} with a
unit $c$ of $U$ an \bfem{artinian supervaluation} (\textbf{with
constant} ~$c$).\end{defn}

The case $c=1$ covers the ultrametric supervaluations (Definition
\ref{defn3.7}), but for the supervaluations $\varphi_L,$ $\psi_L$
with $L$ a true CMC-subset of $R$ we have $c>1.$

Artinian supervaluations abound among \vo-valuations, and they are
a source of prime subsets and CMC-subsets, due to the following
facts.
\begin{prop}\label{prop4.22} If a totally ordered supertropical
semiring $U$ contains a unit $c > e$ then every \vo-supervaluation
$\vrp: R \to U$ is artinian with constant $c$. If in addition $R$
contains a unit $u$ with $\vrp(u) \leq c^{-1}$ then
$$ \mfp_\vrp := \{ a \in R \ds | \vrp(a) < 1 \} $$
is a prime subset of $R$ with exponent $u$ and
$$ A_\vrp := \{ a \in R \ds | \vrp(a) \leq 1 \} $$
is a CMC-subset of $R$ with exponent $u$.
\end{prop}

\begin{proof} If  $a,b \in R$ then
$$ \vrp (a+b) \leq e \vrp(a+b) \leq \max (e \vrp (a)), e \vrp (b) \leq c \max ( \vrp(a), \vrp(b)).$$
If in addition $u \in R$ and  $\vrp(u) = c^{-1}$, then
$$ \vrp (u(a+b)) \leq \max (\vrp(a), \vrp (b)).$$
This gives the claims about  $\mfp_\vrp$ and $A_\vrp$.
\end{proof}

\section{Monotone transmissions and total dominance: Some examples}\label{sec:5}
In the last section we introduced totally ordered supertropical
semirings and then obtained  a large stock of --  as we believe --
natural examples of supervaluations with values in such semirings,
which we called \emph{value ordered supervaluations}, or
\emph{\vo-supervaluations} for short.

They call for a theory of dominance and transmissions adapted to
this special class of supervaluations, which is parallel to our
general theory in \cite{IKR1} and \cite{IKR2}. We give basic steps
of such a theory. Here things seem to be easier than  in the
general theory, since our special transmissions, called
\emph{``monotone transmissions"}, turn out to be semiring
homomorphisms, cf.~ Theorem \ref{thm5.3}.

\begin{defn}\label{defn5.1} Let $U_1$, $U_2$ be totally ordered
supertropical semirings. We call a transmission $\al: U_1 \to U_2$
(as defined in \cite[\S5]{IKR1}) \textbf{monotone} if $\al$ is
order preserving, i.e.,
$$ \forall x,y \in U_1: \quad x \leq y \ \Rightarrow \ \al(x) \leq \al(y).$$

\end{defn}

\begin{example}\label{examps5.2} Let $(\tT_1, \tG_1, v_1)$ and $(\tT_2, \tG_2,
v_2)$ be triples consisting of totally ordered monoids $\tT_i$,
cancellative totally ordered monoids $\tG_i$, and order preserving
homomorphisms $v_i : \tT_i \to \tG_i$ ($i =1,2$). In \S\ref{sec:3}
we associated to such triples ordered supertropical predomains
$$ U_i:= \OSTR(\tT_i, \tG_i, v_i) =\tT_i \dcup \tG_i \dcup \{ 0 \}. $$

Now assume that also order preserving monoid homomorphisms $\bt:
\tT_1 \to \tT_2$ and \\ $\gm: \tG_1 \to \tG_2$ are given with $\gm
v_1 =  v_2 \bt$. Assume in addition that $\gm$ is injective. Then
the map $\al:U_1 \to U_2$ with $\al(0)= 0$,  $\al(x)= \bt(x)$ for
$x \in \tT_1$, $\al(y)= \gm(y)$ for $y \in \tG_1$ is a monotone
transmission.

Indeed, $\al$ clearly obeys the rules TM1-TM5 from
\cite[\S5]{IKR1}, hence is a transmission, and looking at the
rules $(\EO1)$--$(\EO5)$ in \S\ref{sec:3}, which describe the
ordering of the $U_i$, one checks that $\al$ is also order
preserving. Notice that compatibility with the rule $(\EO5)$
demands that~$\gm$ is injective.
\end{example}

Example \ref{examps5.2} gives us, up to isomorphism, all monotone
transmissions between totally ordered supertropical predomains
which map tangible elements to tangible elements.

We state  a fundamental fact about monotone transmissions in
general.

\begin{thm}\label{thm5.3}
Assume that $U$ and $V$ are totally ordered supertropical
semirings and that $\al : U \to V$ is an order preserving map,
which is multiplicative, i.e., $\al(xy) = \al(x) \al(y)$ for any
$x,y \in U$. The following are equivalent.

\begin{enumerate}
    \item[(1)] $\al(0) =0$,  $\al(1) =1$, $\al(e) =e$.
    \item[(2)] $\al$ is a semiring homomorphism.

    \item[(3)] $\al$ is a (monotone) transmission.
\end{enumerate}
\end{thm}

\begin{proof} The implications  $(2) \Rightarrow (3)$ and $(3) \Rightarrow (1)$ are  trivial.

$(1) \Rightarrow (2)$: Given $x,y \in U$, we have to verify that
$\al(x+y) = \al(x) + \al(y)$. We may assume that $ex \leq ey$.

\begin{description}
    \item[Case 1]  $e \al (x) < e \al(y)$, hence $ex < ey$.
\\
    Now $x + y =y,$ $\al(x) + \al(y) = \al(y)$.

    \item[Case 2] $e x  = e y$, hence $e \al (x) = e \al(y)$.
\\
    Now $x + y =ex,$ $\al(x) + \al(y) = e\al(y) = \al(ey)$.
    \item[Case 3] $e x  < e y$, but $e \al (x) = e \al(y)$.
\\
    Now $ex < y < ey$, hence $e \al(y) = e \al(x) \leq \al(y) \leq \al(ey) = e
    \al(y)$,  hence \\ $\al(y) = e\al(y)$. We have $x+y = y$, $\al(x) + \al(y) = e \al(y) =
    \al(y)$.
\end{description}
Thus $\al(x) + \al(y) = \al(x+y)$  in all cases.
 \end{proof}

We  coin a notion of ``\emph{total dominance}" for \vo-valuations
refining the definition of dominance for  arbitrary
supervaluations in \cite[\S5]{IKR1}, and relate this  to monotone
transmissions.

First recall form \cite[\S5]{IKR1} that we defined for
supervaluations $ \vrp: R \to U$, $\psi: R \to V$, that $\vrp$
\emph{dominates} $\psi$, and wrote $\vrp \geq \psi$, if for all
$a,b \in R$,
\begin{align*}
& \D1: \quad   \vrp(a) = \vrp(b) \quad \Rightarrow  \  \psi(a) = \psi(b), \\
& \D2: \quad   e \vrp(a) \leq  e \vrp(b) \Rightarrow  \  e\psi(a) \leq e\psi(b), \\
& \D3: \quad   \vrp(a) \in eU  \ \; \quad \Rightarrow  \  \psi(a)
= eV.
\end{align*}
\begin{defn}\label{defn5.4} Assume that $ \vrp: R \to U$, $\psi: R \to
V$ are \vo-valuations. Then we say that $\vrp$ \textbf{dominates
$\psi$ totally}, and write $\vrp \getot \psi$, if $\vrp$ and
$\psi$ obey axiom D3 but instead of D1, D2 obey the modified
axioms
\begin{align*}
& \D1': \quad   \vrp(a) \leq \vrp(b) \quad \Rightarrow  \  \psi(a) \leq \psi(b), \\
& \D2': \quad   e \vrp(a) =  e \vrp(b) \Rightarrow  \  e\psi(a) =
e\psi(b).
\end{align*}
\end{defn}
Here axiom $\D1'$ is stronger than $\D1$ but $\D2'$ is weaker than
$\D2$. But notice that  $\D1'$ and $\D2'$ together imply $\D2$.
Indeed, if $ e\vrp(a) < e\vrp(b)$, then $ \vrp(a) < \vrp(b)$,
since $ \vrp(a) \geq \vrp(b)$ would imply $ e\vrp(a) \geq
e\vrp(b)$. From this we obtain by $\D1'$ that $\psi(a)  \leq
\psi(b)$ and then $e\psi(a) \leq e\psi(b)$. On the other hand, if
$e\vrp(a) = e \vrp(b)$, then we have  $e\psi(a) = e\psi(b)$ by
$\D2'$, hence again $e\psi(a) \leq e\psi(b)$. We conclude

\begin{remark}\label{rmk5.5}
If $\vrp \letot \psi$, then $\vrp \leq \psi$.
\end{remark}

%

We look for examples of total dominance between the artinian
supervaluations constructed at the end of \S\ref{sec:4} (Examples
\ref{examps4.17}). First an obvious remark.

\begin{remark}\label{rmk5.7}
Assume that  $ \vrp: R \to U$ is an artinian supervaluation with
constant $c$. If $\vrp$  totally dominates $\psi: R \to V$, then
$\psi$ is again artinian with constant $\al_{\psi,\vrp}(c)$. In
particular, if $\vrp$ is ultrametric then  $\psi$ is ultrametric.
\end{remark}

\begin{examples}\label{examps5.8} Let $R$ be a semiring and $L$ a
true  generalized CMC-subset of $L$. Then the artinian
supervaluation
$$\vrp_{L,u} : R \ds \to U(R,L, Q_u(L)) $$
(cf. Example \ref{examps4.17}.a) totally dominates the ultrametric
supervaluation $\vrp_{Q_u(L)}$. Likewise, if $B_n(L) \neq R$, the
artinian supervaluation
$$\psi_{L,u} : R \ds \to U(R,L, B_u(L)) $$
(cf. Example \ref{examps4.17}.b) totally dominates the ultrametric
supervaluation $\vrp_{B_u(L)}$.

\end{examples}

In these examples the associated transmissions restrict to the
identity on the ghost ideals.
\begin{examples}\label{examps5.9} Assume again that $L$ is a true
generalized CMC-subset of a semiring $R$ and that $u$ is an
exponent of~$L$.  We choose some $g \in R^* \cap L$. Then also
$f:= u g$ is an exponent of~ $L$. Let $B := B_u (L)$, $B' := B_f
(L)$, $Q := Q_u (L)$, and $Q' := Q_f (L)$.  We have $u^{-1} =
f^{-1}g \in B'$, $g^{-1} = f^{-1}u \in B'$, and we conclude easily
that $B' \supset B$ and then $B' = \bigcup_ng^{n}B$. We have $$ Q'
= \bigcap_n g^n u^n L \subset \bigcap_n  u^n g^m L = g^mQ$$ for
any fixed $m \in \mathbb N$, hence $Q' \subset \bigcap_n g^n Q$.
On the other hand,
$$ Q' = \bigcap_n g^n u^n B' \supset \bigcap_n  g^n u^n  B = \bigcap_n  g^n  B \supset \bigcap_n  g^n  Q,$$
and we conclude that $ Q' = \bigcap_ng^{n}Q$. For $x\in R $ we
have
 $$[B':x]=\bigcup_n[g^{-n}B:x]=\bigcup_n g^{-n}[B:x].$$

$$[Q':x]=\bigcap_n[g^{n}Q:x]=\bigcap_n g^{n}[Q:x].$$
From these formulas it is evident that for any $x,y \in R$
$$  [B:x] \supset [B:y] \dss \Rightarrow [B':x] \supset [B':y]$$
and
$$  [Q:x] \supset [Q:y] \dss \Rightarrow [Q':x] \supset [Q':y].$$
The second implication tells us that the \m-valuation $v_Q$
dominates $v_{Q'}$. It is now essentially trivial to verify for
$\vrp_{L,u}$ and  $\vrp_{L,f}$ the axioms $\D1'$, $\D2'$, $\D3$.
Thus the artinian supervaluation~ $\vrp_{L,u}$  dominates
$\vrp_{L,f}$ totally.

Likewise, if $B' = R$, the artinian supervaluation $\psi_{L,u}$
dominates $\psi_{L,f}$ totally.
\end{examples}

Notice that in these examples it would not be possible to resort
to the construction in Example \ref{examps5.2} for gaining
directly the appropriate monotone transmissions, since the
transmissions from $v_{B}/C$ to $v_{B'}/C$ and from $v_Q / C$ to
$v_{Q'} / C$ usually are not injective.

We add two examples of total dominance between ultrametric
supervaluations.
\begin{example}\label{examps5.10} Let $\mfp$ be a prime of a semiring
$R$, but not a prime ideal of $R$. Then $A:= A(\mfp) := [\mfp:
\mfp]$ is a proper CMC-subsemiring of $R$. For any $x \in R$ we
have
$$ [A: x] = \big[ [\mfp:\mfp]: x \big] = \big[ [\mfp:x]: \mfp \big].$$
Thus for $x,y \in R$ with $[\mfp:x] \supset [\mfp:y]$ we have
$[A:x] \supset [A:y]$. This tells us that the \VO-valuation
$v_\mfp: R \onto M(R,\mfp)$ introduced in \S\ref{sec:1} dominates
the $V$-valuation $v_A: R \onto M(R,A)$ also introduced there. It
is easily verified that the associated ultrametric supervaluation
$$ \vrp_\mfp : R \to U(R, \mfp)$$  (cf. Example
\ref{examps3.11}.i) totally dominates the ultrametric
supervaluation
$$ \vrp_A : R \to U(R, A)$$ (cf. Example
\ref{examps3.11}.ii) by checking  the axioms $\D1'$, $\D2'$,
$\D3$. Alternatively we can construct explicitly a monotone
transmission $\al: U(R,\mfp) \to U(R,A)$  with $\vrp_A = \al \circ
\vrp_\mfp$, by resorting  to a natural commuting  square of order
preserving semiring homomorphisms
 $$\xymatrix{
     M(R,\mfp)    \ar[d]  \ar[rr]^{\gm}  & &  M(R,A) \ar[d]\\
     M(R,\mfp)/C      \ar[rr]^{\gm/C}  & &  M(R,A)/C\\
  }$$
with $\gm$ the transmission from $v_\mfp$ to $v_A$,   $\gm/C$ the
transmission from $v_\mfp/C$ to $v_A/C$, and canonical vertical
arrows.
\end{example}

\begin{example}\label{examps5.11} Let $A$ be a proper CMC-subsemiring
of $R$ and $$\mfp:= P(A) := \{ x\in R \ds  | \exists y \in R \sm A
: xy \in A\},$$ which is a prime of $R$. We study again the
associated \m-valuations $v_A$ and $v_\mfp$. We have
$$\mfp=  \{ x\in R \ds | v_A (x) < 1 \}.$$ Consequently, given  $x\in R
$, an element $z$ of $R$ lies in $[\mfp: x]$ iff $v_A(x) + v_A(z)
< 1$. Thus for elements $x,y$ of $R$ with $v_A(x) < v_A(y)$ we
have $[\mfp : x ] \supset [\mfp:y]$, hence $v_\mfp(x) \leq
v_\mfp(y)$. This shows that $v_A$ dominates $v_\mfp$. One now can
verify in the same way as in the preceding example that the
ultrametric supervaluation $\vrp_A$ dominates $\vrp_\mfp$ totally.
\end{example}

\section{Order compatible TE-relations}\label{sec:6}
We now look at monotone transmissions via equivalence relations.

\begin{defn}\label{defn6.1} Let $U$ be a totally ordered
supertropical semiring. We call an equivalence  relation $E$ on
the set $U$ an \textbf{order compatible TE-relation}, or
\textbf{OCTE-relation} for short, if the following holds:

\begin{enumerate}
    \item $E$ is multiplicative, i.e., $\forall x,y,z \in E$: $x \sim_E y \Rightarrow xz \sim_E
    yz$.
    \item $E$ is {order compatible}, i.e., obeys the axiom
    (OC) form \cite[\S4]{IKR2}; equivalently, if all $E$-equivalences classes
    are convex subsets of the totally ordered set $U$.
\end{enumerate}
\end{defn}

\emph{Comment.} This terminology is related to the definition of
``TE-Relations"  (= transmissive equivalence relation) in
\cite[\S4]{IKR2}. To repeat, an equivalence relation $E$ on a
supertropical semiring $U$ with ghost ideal $M := eU$ is called a
\textbf{TE-relation}, if $E$ is multiplicative (Axiom TE1 in
\cite[\S4]{IKR2}), the restriction $E | M$ is order compatible
(Axiom TE2), and $x \in U $ with  $ex \sim_E 0$ is itself
E-equivalent to $0$ (Axiom TE3).

If $U$ is a totally ordered and $E$ is an OCTE-relation on $U$,
then clearly TE1 and TE2 are valid. But also TE3 holds: If $x \in
U$ then $0 \leq x \leq ex$, and thus $ex \sim_E 0$ implies $x
\sim_E 0$ since~$E$ is order compatible. Thus an OCTE-relation on
$U$ is certainly a  TE-relation. \endbox

\pSkip

If $E$ is an OCTE-relation, then it is obvious that the set $E/U$
of $E$-equivalence classes has a (unique) well defined structure
of a totally ordered monoid, such that the map $$\pi_E  : Y \to
U/E, \qquad x \mapsto [x]_E,$$ is multiplicative and order
preserving (cf. the arguments in the beginning of
\cite[\S4]{IKR2}).
%
This structure is given by
the rules ($x,y \in U$) \begin{equation}\label{eq:6.1} [x]_E \cdot
[y]_E = [xy]_E,
\end{equation}
 \begin{equation}\label{eq:6.2} [x]_E \leq [y]_E  \dss
\Leftrightarrow x \leq y.
\end{equation}
The monoid $U/E$ has the unit element $[1_U]_E$ and the absorbing
element $[0_U]_E$. Further it is easily checked that the map
$$ p: U/ E \ds \to U/E, \qquad [x]_E \mapsto [ex]_E,$$
obeys the rules $(\Gh1)$--$(\Gh5)$ form \S3. Thus we have on $U/E$
the structure of a totally ordered supertropical semiring, as
indicated in Theorem \ref{thm3.3}.

The ghost ideal of $U/E$ is
$$ p(eU) = \{ [x]_E \ds | x \in M \}.$$
Its unit element $e_{U/E}$ is the class $[e]_E$. We see that the
map $\pi_E: U \to U/E$ fulfills condition (1) in Theorem
\ref{thm5.3}, and we  conclude by that theorem that $\pi_E$ is a
monotone transmission. Thus for $x,y \in U$ we have the rule
\begin{equation}\label{eq:6.3}
    [x]_E + [y]_E = [x+y]_E.
\end{equation}
We have arrived at the following theorem.
\begin{thm}\label{thm6.2}
If $E$ is an OCTE-relation on a totally ordered supertropical
semiring $U$, then the set $U/E$ carries a (unique) structure of a
totally ordered supertropical semiring such that the map $\pi_E: U
\to U/E$ is a monotone transmission, and hence a semiring
homomorphism. The structure of the  ordered supertropical semiring
$U/E$ is given by the rules \eqref{eq:6.1}-\eqref{eq:6.3}.
\end{thm}

\begin{rem} Conversely, if $\al: U \onto V$ is a surjective monotone
transmission, then the equivalence relation $E := E(\al)$ on $U$
(i.e., $x \sim_E y$ iff $\al(x) = \al(y)$) is clearly an
OCTE-relation, and $\al = \rho \circ \pi_E$ with $\rho$ an
isomorphism from  $U/E$ to the semiring $V$. Thus knowing the
OCTE-relation on $U$ gives us a hold on all  surjective monotone
transmissions starting from $U$. \endbox
\end{rem}

In \cite[\S4]{IKR2} we pursued the question, when a given
TE-relation $E$ on a supertropical semiring $U$ is
\textbf{transmissive}, i.e., the monoid $U/E$ admits the structure
of a supertropical semiring such that $\pi_E$ is  transmission.
The main result there \cite[Theorem 4.7]{IKR2} stated that this
happens if the monoid $(M/E) \sm \{0\}$ is cancellative ($M: =
eU$). The present Theorem~\ref{thm6.2} exhibits another class of
TE-relations, which are transmissive. Here no cancellation
hypothesis is needed.

\pSkip

 In \cite{IKR2} a transmissive equivalence relation $E$ on
supertropical semiring is called \textbf{homomorphic}, if the
transmission $\pi_E: U \to U/E$ is a semiring homomorphism. \S5
and \S6 of \cite{IKR2} contain various explicit examples of
homomorphic equivalence relations.

Assume now that $U$ is  totally ordered. As we know,  all
OCTE-relations on $U$ are homomorphic. We  search for cases where
the homomorphic equivalence relations described in \cite[\S5 and
\S6]{IKR2} are OCTE-relations.

\begin{example}\label{examps5.15}
If $\mfa$ is any ideal of $U$ then the homomorphic equivalence
relation $E(\mfa)$ on $U$ (cf. \cite[\S5]{IKR2}) is  order
compatible, hence is an OCTE-relation.
\end{example}

\begin{proof}
We prove that the $E(\mfa)$ equivalence classes are convex in $U$
and then will be  done. Let $x,y,z \in U$ be given with $x \leq z
\leq y$ and $x \sim_\mfa y$. We verify that $y \sim_\mfa z$ by
using the description of $E(\mfa)$ in \cite[Theorem 5.4]{IKR2}. We
have $ex \leq ez \leq ey$.
\begin{description}
    \item[{Case 1}] $ex \leq ea$ for some  $a \in \mfa$. Then $x \sim_\mfa
    0$, hence $y \sim_\mfa 0 $. Thus $ey \leq eb$ for some $b \in \mfa$, and from $ez \leq
    ey$ we infer that $z \sim_\mfa 0 $.
    \item[{Case 2}] $ex > ea$ for every  $a \in \mfa$.
Now $x = y$, hence also $x=z$.
\end{description}
Thus  $z \sim_\mfa
    x$ in both cases.
\end{proof}
\begin{example}\label{examps5.16} Let $\Phi$ be a homomorphic equivalence
relation on $M := eU$, and let $\mfA$ be an ideal of $U$
containing $M$. We study the equivalence relations $E := E(U,
\mfA, \Phi)$ defined in \cite[\S6]{IKR2} (cf. Definition 6.11).
Thus for $x,y \in \mfA$,
$$ x \sim_E y  \quad \Leftrightarrow \quad \left\{
\begin{array}{lll}
  \text{either} &  &  x=y, \\[1mm]
  \text{or} &   & x,y \in \mfA \ \text{and} \ ex \sim_\Phi ey. \\
\end{array}%
\right.$$ We know from \cite[Theorem 4.13.i]{IKR2} that $E$ is
multiplicative. It is easily verified that every equivalence class
of $E$ is convex in $U$ if the following two conditions hold:
\begin{enumerate}
    \item $\inu_U(x) \subset \mfA$ for every $x \in M$ such that $x \sim_\Phi
    y$ for some $y\in M $ with $y < x$.
    \item  $\inu_U(x) \cap \mfA$ is convex in $\inu_U(x)$ for the
    other $x \in M$ \footnote{This means that $\nu_u^{-1} (x) \cap \mfA $
    is an upper set in the totally ordered set $\nu^{-1}_u(x)$.}.
\end{enumerate}
Looking at \cite[Theorem 6.14]{IKR2} we see that the conditions
(1), (2) imply that $E$ is homomorphic. Thus $E$ is an
OCTE-relation precisely if these conditions (1), (2) are valid.
\end{example}

The case that $\Phi = \diag(M)$, i.e., $x \sim_\Phi y $ iff $x=y$,
gives us
\begin{sexample}\label{examps5.17} Let $\mfA$ be an ideal of $U$
containing $M$. The MFCE-relation $E(U, \mfA)$ (cf.
\cite[Defintion 6.15]{IKR2}) is order compatible (hence an
OCTE-relation) iff for every $x \in M$ the set $\mfA \cap
\inu_U(x)$ is convex in  $\inu_U(x)$.
\end{sexample}


\begin{thebibliography}{IMS}


 \bibitem[A1]{A1} E. Artin, \textit{Algebraic Numbers and Algebraic
 Functions.
 Lecture Notes by I. Adamson}, Gordon and Breach, New York, 1967.

  \bibitem[A2]{A2} E. Artin, \textit{Theory of Algebraic Numbers.
  Notes by G. W\"urges}, translated by G. Striker, G\"ottingen,
  1956.

 \bibitem[B]{B} N. Bourbaki, \textit{Alg. Comm. VI}, \S3, No.1.


\bibitem[BCR]{BCR} J. Bochnak, M. Coste, and M-F. Roy, \textit{Real Algebraic
Geometry}, Springer, 1998.

 \bibitem[C]{C} I.G. Connell, \textit{A natural transformation of the Spec functor}, J. Algebra \textbf{10} (1968), 69-91.

\bibitem[Gr]{Gr} M. Griffin, \textit{Generalizing valuations to commutative rings},
  Queen's Mathematical Preprint No.~1970-40, Queen's University,
  Kingson, Ontario, Canada.

   \bibitem[HV1]{HV1} D.K.
Harrison and M.A. Vitulli, \textit{$V$-valuations of a commutative
ring I}, J. Algebra   \textbf{126} (1989), 264--292.

 \bibitem[HV2]{HV2} D.K.
Harrison and M.A. Vitulli, \textit{Complex-valued places and CMC
subsets of a field}, Commun. Algebra \textbf{17} (1989),
2529--2537.


  \bibitem[HK]{HK} R. Huber and M. Knebusch, \textit{On valuation spectra},
   Contemp. Math. \textbf{155} (1994), 167--206.



 \bibitem[IMS] {IMS} I. Itenberg, G. Mikhalkin, and E.
Shustin, \textit{Tropical Algebraic Geometry}, Oberwolfach
Seminars, 35, Birkh\"auser Verlag, Basel, 2007.


 \bibitem[I]{I} Z. Izhakian, \textit{Tropical arithmetic and tropical matrix algebra},
  Commun. Algebra  \textbf{37}:4 (2009), {1445--1468}. (Preprint at arXiv:
  math/0505458v2.)



\bibitem[IKR1]{IKR1}  Z. Izhakian, M. Knebusch, and L. Rowen, \textit{Supertropical semirings and
supervaluations}, J. Pure and Applied Alg., to appear. (Preprint
at arXiv:1003.1101.)

\bibitem[IKR2]{IKR2}  Z. Izhakian, M. Knebusch, and L. Rowen,
\textit{Dominance and transmissions in supertropical valuation
theory}, preprint at arXiv:1102.1520v1, 2011.

\bibitem[IKR3]{IKR3}  Z. Izhakian, M. Knebusch, and L. Rowen, \textit{Supervaluations
on semifields}, in preparation.

\bibitem[IR1]{IzhakianRowen2007SuperTropical}
Z.~Izhakian and L.~Rowen.
\newblock \textit{Supertropical algebra}, Advances in Math., 225:2222–--2286, 2010.
\newblock (Preprint at arXiv:0806.1175.)

\bibitem[IR2]{IzhakianRowen2008Matrices}
Z.~Izhakian and L.~Rowen,
\newblock \textit{Supertropical matrix algebra}.
\newblock  Israel J. Math.,  182(1):383--424, 2011.
\newblock (Preprint at arXiv:0806.1178.)



\bibitem[IR3]{IzhakianRowen2009Equations}
Z.~Izhakian and L.~Rowen,
\newblock \textit{Supertropical matrix algebra II: Solving tropical
equations},
\newblock  Israel J. Math., to appear.
\newblock (Preprint at arXiv:0902.2159, 2009.)

\bibitem[IR4]{IzhakianRowen2009Resulatants}
Z.~Izhakian and L.~Rowen,
\newblock {\em Supertropical polynomials and resultants.}
\newblock { J. Alg.}, 324:1860--1886, 2010. (Preprint at arXiv:0902.2155.)

  \bibitem[KZ1]{KZ1} M. Knebusch and D. Zhang, \textit{Manis Valuations and Pr\"ufer Extensions.
   I.
  A New Chapter in Commutative Algebra}, Lecture Notes in Mathematics, 1791, Springer-Verlag,
   Berlin, 2002.

 \bibitem[KZ2] {KZ2} M. Knebusch and D. Zhang, \textit{Convexity, valuations, and
Pr\"ufer extensions in real algebra}, Doc. Math. \textbf{10}
(2005), 1--109.


\bibitem[ML] {ML} S. Maclane, \textit{Categories for the working mathemtician}, 4th ed.
Springer Verag, 1998.


\bibitem[S]{S} P. Samuel, \textit{ La notion de place dans
un anneau}, Bull. Soc. Math., France 85 (1957), 123-133.

 \bibitem[VV]{VV} K.G. Valente and M.A. Vitulli, \textit{Complex-valued places and CMC
subsets of a ring}, Commun. Algebra   \textbf{18} (1990),
3743--3757.

 \bibitem[Z]{Z} D. Zhang, \textit{The $M$-valuation spectrum of a commutative
ring}, Commun. Algebra \textbf{30} (2002), 2883--2896.

\end{thebibliography}
\end{document}